\documentclass[11pt,twoside,reqno,a4paper]{amsart}
\usepackage[margin=2cm]{geometry}
\usepackage{amssymb}
\usepackage{color}
\usepackage{mathtools}
\usepackage{tikz}
\usepackage{xspace}
\usepackage{comment}

\usepackage[backend=biber,
style=trad-abbrv,
            backref=true,
            backrefstyle=three,
            hyperref=true,
	    url=false,
	    isbn=false,
	    doi=false,
	    sortcites]{biblatex}
\usepackage[]{hyperref}
\hypersetup{
  colorlinks=true,
  linkcolor=black,
  anchorcolor=black,
  citecolor=black,
  filecolor=black,      
  menucolor=red,
  runcolor=black,
  urlcolor=black,
}
\usepackage[capitalize,noabbrev,nameinlink]{cleveref}

\numberwithin{equation}{section}
\crefformat{equation}{(#2#1#3)}
\crefformat{figure}{#2Figure #1#3}
\crefformat{enumi}{(#2#1#3)}
\crefrangeformat{enumi}{(#3#1#4)--(#5#2#6)}
\crefformat{enumii}{(#2#1#3)}
\crefformat{enumiii}{(#2#1#3)}
\crefformat{theorem}{#2Theorem~#1#3}
\crefformat{lemma}{#2Lemma~#1#3}
\crefformat{proposition}{#2Proposition~#1#3}
\crefformat{corollary}{#2Corollary~#1#3}
\crefformat{claim}{#2Claim~#1#3}

\newcommand{\ie}{\textit{i.e.}\@\xspace}
\newcommand{\cf}{\textit{i.e.}\@\xspace}

\numberwithin{equation}{section}

\theoremstyle{plain}
\newtheorem{theorem}{Theorem}[section]

\newtheorem{conjecture}[theorem]{Conjecture}

\theoremstyle{remark}
\newtheorem{remark}[theorem]{Remark}

\theoremstyle{definition}

\newtheorem{question}[theorem]{Question}

\newcommand{\cH}{\mathcal{H}}
\newcommand{\cL}{\mathcal{L}}
\newcommand{\cM}{\mathcal{M}}
\newcommand{\bbR}{\mathbb{R}}

\newcommand{\bbE}{\mathbb{E}}

\newcommand{\bbN}{\mathbb{N}}
\newcommand{\bbP}{\mathbb{P}}

\newcommand{\bi}{\mathbf{i}}
\newcommand{\bj}{\mathbf{j}}

\newcommand{\eps}{\varepsilon}
\newcommand{\dimh}{\dim_{\mathrm{H}}}
\newcommand{\dima}{\dim_{\mathrm{A}}}
\newcommand{\dimb}{\overline{\dim}_{\mathrm{B}}}
\newcommand{\dimaspec}[1]{{\dim}_{\mathrm{A}}^{#1}}
\newcommand{\dimPhi}[1]{{\dim}^{#1}}

\DeclareMathOperator{\hdim}{dim_H}
\DeclareMathOperator{\bdim}{dim_B}

\DeclareMathOperator{\Bin}{Bin}

\DeclareMathOperator{\diam}{diam}
\DeclareMathOperator{\proj}{proj}

\DeclareMathOperator{\por}{por}

\begin{document}

\title{Recent Progress on Fractal Percolation}

\author{Istv\'an Kolossv\'ary}
\address[Istv\'an Kolossv\'ary]
{
  HUN-REN Alfr\'ed R\'enyi Institute of Mathematics\\
  Reáltanoda u.\ 13--15\\
  HU 1053 Budapest\\
  Hungary
}
\email{kolossvary.istvan@renyi.hu}

\author{Sascha Troscheit}
\address[Sascha Troscheit]
{
  Department of Mathematics\\
  University of Uppsala\\
  Box 480\\
  753 24 Uppsala, Sweden
}
\email{sascha.troscheit@math.uu.se}

\date{\today}

\begin{abstract}
This is a survey paper about the fractal percolation process, also known as Mandelbrot percolation.
It is intended to give a general breadth overview of more recent research in the topic, but also
includes some of the more classical results, for example related to the connectivity properties.
Particular emphasis is put on the dimension theory of the limiting set and also on the geometry of
the non-trivial connected components in the supercritical regime. In particular, we show that both the Assouad
spectrum and intermediate dimensions of the non-trivial connected components are constant equal to
its box dimension despite its Hausdorff, box and Assouad dimensions known to being distinct.
\end{abstract}

\maketitle

\section{Introduction}\label{sec:Intro}

Fractal percolation is a model of fractality that is stochastically self-similar. In its most basic
form it is obtained by iteratively dividing the square $[0,1]^2$ into $2\times 2 = 4$ congruent
subsquares and keeping each subsquare independently with probability $p\in(0,1)$. These subsquares
are then divided again and kept independently with probability $p$ ad infinitum. For reasonable
choices of $p$ (here $p>\tfrac14$) there exists positive probability that we end up with a limiting
set, called \emph{Mandelbrot percolation}.
See Figure~\ref{fig:samplePercolation} for realisations with different $p\in(\tfrac14, 1)$.
This model of sets with stochastic self-similarity was first proposed by Mandelbrot in the context of
fluid turbulence \cite{Mandelbrot_CanonicalCurdling74,Mandelbrot_BookFravtalGeomNature82}, and over
the past few decades much research has gone into studying extensions of this model, generally
referred to as \emph{fractal percolation}. The interest in fractal percolation arises in a multitude
of contexts, such as topology (connectedness), metric geometry (dimension), geometric measure theory
(projections, slices), or statistical physics (percolation), amongst many other.

Given the interest of the mathematical community, it comes to no surprise that there are already a
wealth of surveys available, not least in the Fractal Geometry and Stochastics conference
proceedings. Of note are the surveys by
Chayes \cite{Chayes_FGSSurvey95} from 1995, as well as those of Rams and Simon
\cite{RamsSimon_FractalPercSurvey14} (2014) and Simon and V\'ag\'o
\cite{SimonVago_FractalPercSurvey18} (2018). We refer the reader to these articles for some of the
more classical literature and highlight in this survey instead the general breadth of recent
research.
More recently (2025) Orgov\'anyi and Simon \cite{OrgovanyiSimon_Survey25} wrote a
survey on fractal percolation, focusing on projections and non-empty interior.
In this survey we focus instead on recent progress on the dimension theory as well as the connectedness
properties of the resultant percolation set. We also determine the
intermediate dimension and Assouad spectrum of the complement of the totally disconnected part of Mandelbrot
percolation.
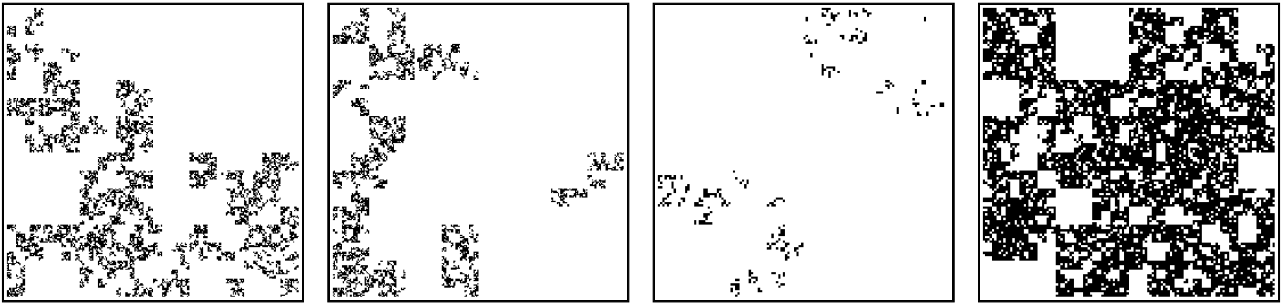
\begin{figure}[t]
  \begin{tikzpicture}
    \filldraw (-.03,-0.03) rectangle (3.9+.03,3.9+.03);
    \node () at (1.95,1.95) {\includegraphics[width=3.9cm]{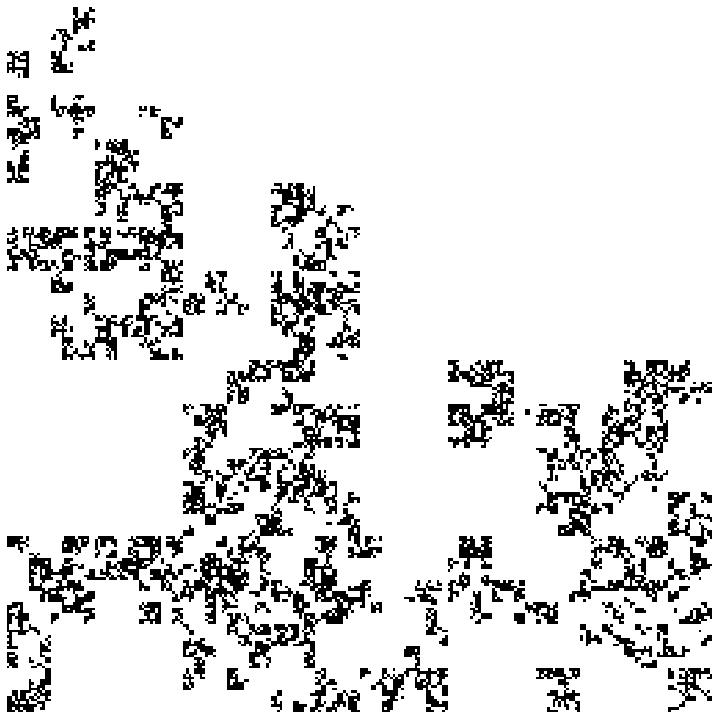}};
  \end{tikzpicture}
  \hfill
  \begin{tikzpicture}
    \filldraw (-.03,-0.03) rectangle (3.9+.03,3.9+.03);
    \node () at (1.95,1.95) {\includegraphics[width=3.9cm]{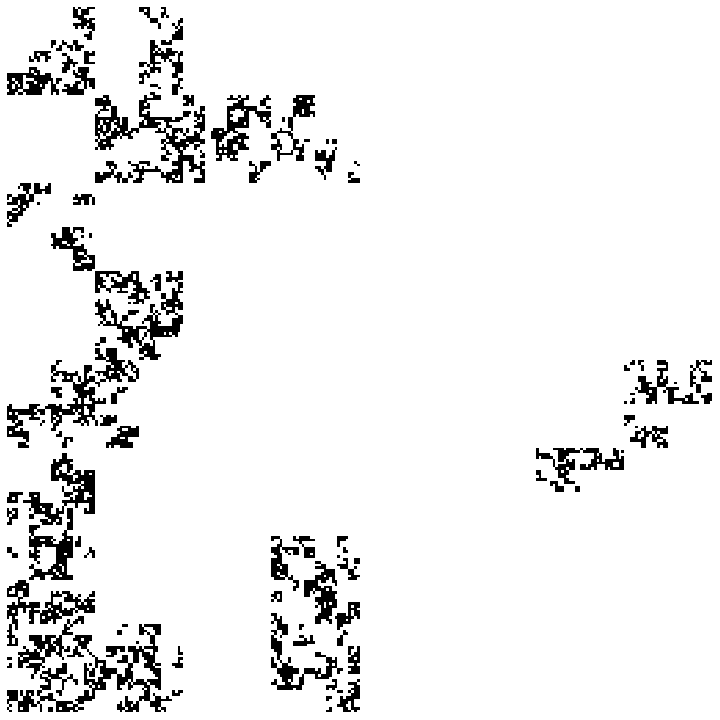}};
  \end{tikzpicture}
  \hfill
  \begin{tikzpicture}
    \filldraw (-.03,-0.03) rectangle (3.9+.03,3.9+.03);
    \node () at (1.95,1.95) {\includegraphics[width=3.9cm]{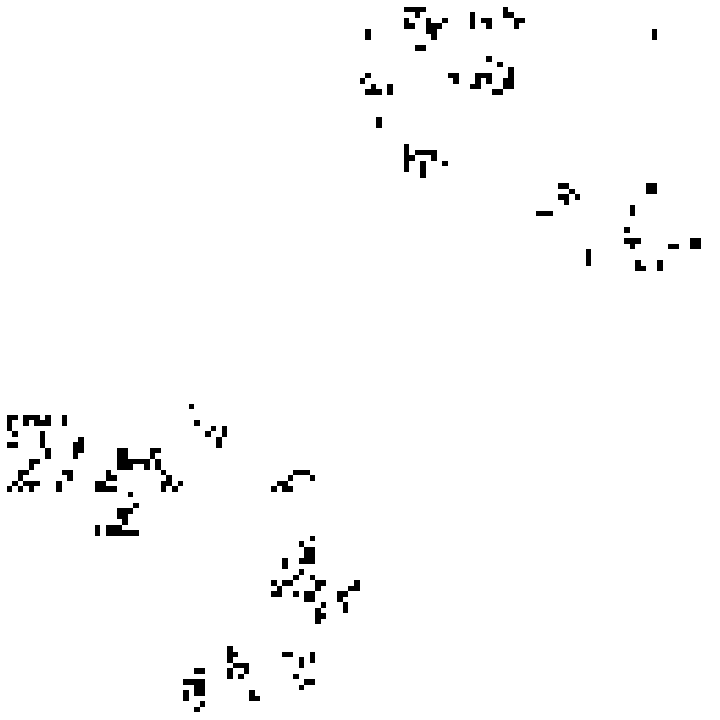}};
  \end{tikzpicture}
  \hfill
  \begin{tikzpicture}
    \filldraw (-.03,-0.03) rectangle (3.9+.03,3.9+.03);
    \node () at (1.95,1.95) {\includegraphics[width=3.9cm]{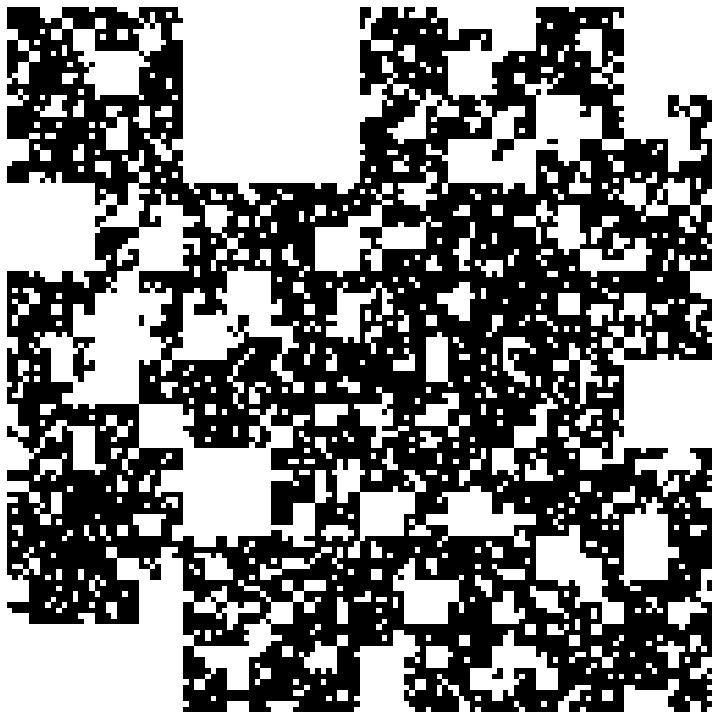}};
  \end{tikzpicture}
  \hfill
 
  \caption{Realisations of fractal percolation with $N=2$, $d=2$, $p=\tfrac34$ (left, centre left),
  $p=\tfrac12$ (centre right), $p=\tfrac78$ (right).}
  \label{fig:samplePercolation}
\end{figure}

\subsection{The model}
We first set up the standard Mandelbrot percolation, as informally described in the introduction. We
then generalise it slightly to a model that is sufficient for handling the generality of most
of the remaining article.
Later, we consider more general constructions, but to avoid cluttering up notation, we
leave setting up proper notation for these until then.

\subsubsection{Mandelbrot percolation}\label{sec:MandelbrotPerc}
The classical Mandelbrot percolation has three parameters. 
Let $d\geq 1$ be the ambient space dimension and fix $N\geq2$, the number of subdivisions.
Further let $p\in[0,1]$. 
Let $[0,1]^d\subset \bbR^d$ be the unit cube. Let $X$ be a random variable that takes value $1$ with
probability $p$ and value $0$ with probability $1-p$.
Let $\Sigma_k$ be the collection of all congruent subcubes of $[0,1]^d$ with sidelength $1/N^k$ to get
a partition of the space with $N^{dk}$ elements. 
Let $X_{\bi}:\Sigma_k \to \{0,1\}$ be independent random variables, indexed by $\bi\in\Sigma_k$ that have
the same distribution as $X$.
Write $Q_k=\{\bi\in \Sigma_k \mid X_{\bi} = 1\}\subseteq \Sigma_k$ for random collection of cubes that are
each kept with probability $p$ independently.
For the remainder, we will ignore the underlying probability space $\Omega$, but it can be modelled as a
countable product space indexed by the cubes $\Sigma_k$, $k\in \bbN$.
We define the limit set $F(\omega) = \bigcap_{k=0}^\infty Q_k(\omega)$ to be the random intersection of
the surviving subcubes. 
Here, as later, we will suppress the underlying realisations, but stress that the sets depend on
some realisation $\omega\in\Omega$. We will denote the induced probability measure on $\Omega$ by
$\mathbb{P}$, or $\mathbb{P}_p$, if we want to highlight its dependence on the construction parameter
$p$.

The model introduced here is fairly straightforward to analyse and one can see that the cardinality
of $\#Q_k$ has binomial distribution $\Bin(N^{dk},p)$. For a point to survive in the limit set $F$,
it must be retained at every step $k$ in the process. A moment's thought gives that the number of
subcubes $\bi j \in Q_{k+1}$, where $\bi \in Q_k$ has distribution $\Bin(N^d,p)$. In fact, this is
an offspring distribution and the number of sets $Z_n$ in $F_n\coloneqq\bigcap_{k=1}^n Q_k$ is a simple Galton--Watson
process with that offspring distribution.
Standard techniques now imply that the process is critical when $p = N^{-d}$ and the process dies
out, leaving $F = \varnothing$, almost surely if $p\leq N^{-d}$.
When $p>N^{-d}$, the set $F_n$ contains in expectation $m^k=(pN^{-d})^k$ subcubes and normalising
the Galton-Watson process $W_n = m^{-n}Z_n$ one obtains an $L^2$ martingale. 
Using martingale convergence theorems, one obtains $W_n \to W \in [0,\infty)$ almost surely, where $W=0$
represents extinction. The probability of extinction can also be computed and is given by the unique $q\in
(0,1)$ such
that
\begin{equation}\label{eq:10}
q = \sum_{k=0}^{N^d}\binom{N^d}{k} p^k (1-p)^{N^d-k} q^k =(1-p+pq)^{N^d}.
\end{equation}

\subsubsection{General Construction}\label{sec:generalConstr}
To arrive at a much more general model, we first abstract the sequence of nested cubes to codings
and construct a symbolic percolation set first. We then project this symbolic set into the geometry
of $\bbR^d$, which allows broad flexibility in how we map the symbolic set into Euclidean space. 

From now we write $\Sigma_k = \{1,\dots, K\}^k$ for the symbolic space consisting of $K^k$ sequences
of length $k$ and $\Sigma = \{1,\dots,K\}^{\bbN}$ for the collection of all infinite words.
The set $\Sigma_0=\{\emptyset\}$ contains the empty word $\emptyset$ only and we
write $\Sigma_* = \bigcup_k \Sigma_k$ for the set of all finite words. One may think of $\Sigma_*$
as the $K$-ary tree with boundary given by $\Sigma$.
We define a projection $\pi:\Sigma_k\to\bbR^d$ in terms of compositions of functions that depend on
the coding.
Let $f_i:[0,1]^d\to [0,1]^d$ be a strict contraction for each $i\in\Sigma_1$. We write $f_{\bi} =
f_{i_1}\circ f_{i_2} \circ \dots \circ f_{i_k}$ for words $\bi=i_1 i_2 \dots i_k \in\Sigma_k$.
Given an infinite word $\bi\in\Sigma$, we define
\[
  \pi(\bi) = \lim_{n\to\infty} f_{\bi|_n}(x) = \lim_{n\to\infty} f_{i_1}\circ\dots\circ f_{i_n}(x),
\]
where $x\in [0,1]^d$ is arbitrary and can be fixed to $x=0$ without loss of generality.
For finite words, we define $\pi(\bi)$ to be the image of the unit cube under the composition of
mappings. That is, for $\bi\in\Sigma_k$,
\[
  \pi(\bi)  = f_{\bi}([0,1]^d) = f_{i_1} \circ f_{i_2} \circ \dots \circ f_{i_k}([0,1]^d)
\]
Similar to above, we let $X_{\bi}$ be a random variable associated with $\bi\in\Sigma_*$ that takes
value 1 with probability $p_{\bi}\in[0,1]$ and $0$ with probability $1-p_{\bi}$, independently of all other
$\bj\neq\bi$ in $\Sigma_*$.
Analogously, we define $Q_k = \{\bi\in\Sigma_k : X_{\bi}=1\} \subseteq \Sigma_k$ to be the retained
codings at construction level $k$. The projection $\pi$ maps the symbolic space
into Euclidean space and we define 
\begin{equation}\label{eq:11}
  F_n = \bigcap_{k=0}^n \bigcup_{\bi\in Q_k} \pi(\bi) =
\bigcap_{k=0}^n \pi (Q_k)
\quad
\text{and}
\quad
F = \bigcap_{n=0}^\infty F_n = \bigcap_{k=0}^\infty \pi(Q_k).
\end{equation}
This more general approach through mappings, called an iterated function system (IFS), allows for
much greater generality. The mappings $f_i$ can now be any tiling of the original cube, allow for
overlaps, as well as doing percolation on Cantor sets and other dynamical attractors.
We refer to this more general model as \emph{fractal percolation} to emphasise this generality.
While we will generally suppress the underlying probability space $\Omega$, one can easily verify
that it is obtained by the countable product of the simple  indicators $X_{\bi}$ over
$\bi\in\Sigma_*$. This general percolation model was, to the best of our knowledge, first described
by Falconer and Miao \cite{FalconerMiao10_RandomSubsets} where they investigated projection
questions of fractal percolation. 
At this point we also mention that fractal percolation is a special case of stochastically
self-similar sets. These sets were first, independently, introduced in the 80s by
Falconer~\cite{Falconer1986_RandomFractals}, Graf~\cite{Graf87_StatSelfSimilar}, also
Mauldin and Williams~\cite{MauldinWilliams86_RandomRecursive}. More recent extensions of
these random recursive constructions are beyond the scope of this survey. We refer
to~\cite{Durand_RndFractalsTreeIndexedMCs,LiuWenWu_GenRandRecursive} and references therein for
additional background.

A few particular examples are the following. 
\begin{enumerate}
\item The Mandelbrot percolation of \cref{sec:MandelbrotPerc} can be recovered from this general
  model by letting $f_{i}$, $i\in \{1,\dots, N^d\}$ be the homotheties that map the unit cube
  $[0,1]^d$ into the $N^d$ congruent subcubes of sidelength $1/N$ parallel to the coordinate-axis
  and setting $p_{\bi} = p$ for all $\bi\in\Sigma_*$. 
\item By choosing $p_{\bi}=0$ for some $\bi\in\Sigma_*$, one can obtain a random subset of a
  deterministic fractal. For example, to obtain the \textit{Random Sierpi\'nski Carpet} studied by
  Dekking and Meester~\cite{DekkingMeester_StructureFractalPerc} consider the IFS defining
  Mandelbrot percolation for $d=2$ and $N=3$ but leave out the map which takes the unit square into
  the middle square. If this map corresponds to symbol $j\in \Sigma_1$, then this is equivalent to
  setting $p_{\bi j}=0$ for all $\bi\in\Sigma_*$. 
\item\label{ex:fatperc} Alternatively, one can also vary the retention probability by iteration
  level, \ie given a sequence $p^{(n)}$, we set $ p_{\bi}=p^{(n)}$ for all $\bi\in\Sigma_n$. It is
  immediate that the expected Lebesgue measure of $F$ is the product $\prod_{n=1}^\infty p^{(n)}$.
  If this product is strictly positive and $p^{(n)}$ is non-decreasing,  Broman \textit{et al.}
  coined it \textit{fat fractal percolation} in~\cite{BromanEtal_FatFractalPerc},
  see~\cref{thm:FatPerc} for results in this setting. 
\item Additional spatial in-homogeneity can be introduced by considering not just similarity
  mappings. For example, assume the IFS consists of contracting strictly self-affine maps of the
  form $f_i(x) = M_i x +t_i$ where the diagonal matrices $M_i$ and the translations $t_i$ are chosen
  such that $\{\pi(i):\, i\in\Sigma_1\}$ tile the unit cube. These planar carpet and higher
  dimensional sponge models have an extensive literature in the deterministic setting, but have also
  attracted considerable attention in the fractal percolation setting, see~\cref{sec:DimStochSelfAffine} for details.
\end{enumerate}
For simplicity, we chose the same  underlying symbolic space but many generalisations are possible.
One could allow the number of
children of a vertex $\bi$ of the tree $\Sigma_*$ to depend on $\bi$ and consequently also let the
corresponding maps depend on $\bi$. This direction is beyond the scope of this article but we refer
the reader to \cite{Chen17_RandomCantor} which studies a percolation model where the number of subcubes at
every level are distinct. Further, the introduction of dependencies leads to models such as
$V$-variable sets and code-tree fractals, that differ from fractal percolation by introducing
(well-motivated) construction level dependencies. 
These are studied, amongst others, in
\cite{BarnsleyHutchinsonStenflo2012,JarvenpaaJarvenpaaWuWu2017,Troscheit2021_CodeTrees} and we refer
the reader to these texts and references
therein.

\subsubsection{Spatial Invariance properties}
One crucial property of stochastically self-similar sets, and fractal percolation in particular, is
a spatial geometric invariance property. Assume that 
the first level images of the open unit cube $f_i((0,1)^d)$ are disjoint, then
these regions now have the same distribution as fractal percolation in the original cube, up to the
scaling provided by $f_i$. If $f_i$ are similarities we will refer to this as stochastically
self-similar, if they are affine, we say it is stochastically self-affine.

Other invariances in distribution exist, usually based on spatial Poisson point
processes, see for
example~\cite{BromanCamia_UniversalConnectivityFractalPercModels,
BromanEtal_FractalCylinderProcess,NacuWerner_RndSoupsCarpetsFractalDims}.
We will
briefly discuss spatially independent martingales of~\cite{ShmerkinSuomala_SIMartingales} in~\cref{sec:generalProcesses}.

\subsection{Overview of contents}

In~\cref{sec:BasicConnProp} classical results on the plane about the percolation threshold are
presented together with a list of recent advances in higher dimensions, on more general $N\times M$
grid and concerning fat fractal percolation. Different notions of dimension are introduced
in~\cref{sec:DimLimitSet}, where we summarise the dimension theory of Mandelbrot percolation and
highlight what is known about more general stochastically self-affine sets.
Section~\ref{sec:ConnectedComp} is dedicated to the geometry of the non-trivial connected components
in the supercritical regime. We give a sketch of the argument of Broman 
\textit{et al.}~\cite{BromanCamiaJoostenMeester_DimConnectedComponent} of a dimension gap between the Hausdorff
and box dimension of the connected components, moreover, we prove results about the Assouad spectrum
and intermediate dimensions of the connected components and survey works on the regularity
properties of the percolating paths. Section~\ref{sec:Dust} gives a short account of recent results
concerning projections, slices, visible part and porosity amongst others.
In~\cref{sec:Generalisations} we give a brief glimpse into various generalisations of the model. We
finish~\cref{sec:OPenProblems} with a few open problems.

\section{Connectivity properties}\label{sec:BasicConnProp}

For most of this section we assume the setting of Mandelbrot percolation introduced
in~\cref{sec:MandelbrotPerc}. We have already seen in~\cref{eq:10} that there is a positive
probability of the limit set $F$ being non-empty if $p>N^{-d}$. In what follows, we always assume
$p>N^{-d}$ and condition on the event that $F\neq \varnothing$.

An immediate observation is that any fixed countable set intersects $F$ with zero probability. This
is simply because each point of $[0,1]^d$ has at most $2^d$ codings on the symbolic space and the
projection of each coding is an element of $F$ with probability 0. 

At a heuristic level it is reasonable that $F$ becomes ``larger'' as the retention probability
grows. For $p$ closer to $N^{-d}$, one may expect $F$ to be totally disconnected. As $p$ grows,
connected components start to appear and for some $p$ closer to 1 there is a component which
connects opposing faces of $[0,1]^d$. It turns out that the transition in connectedness is actually
very abrupt. To formalise this, we introduce some notation. 

From now we consider the planar case, \ie when $d=2$. For any set $E\subseteq[0,1]^2$ and point
$x\in E$, let $E_x$ be the set of points $y\in E$ that are connected to $x$ in $E$.
We call $E_x$ the connected component of $x$ in $E$. We partition $E$ into  
\begin{equation*}
  E^d \coloneqq \{x\in E:\,  E_x=\{x\}\} \text{ and }  E^c\coloneqq  E\setminus E^d,
\end{equation*}
where $d$ refers to `dust' and $c$ to `connected'. We say that $F$ percolates if there exists a
component of $F^c$ which connects the left hand side $L=\{0\}\times[0,1]$ with the right hand side
$R=\{1\}\times[0,1]$ of $[0,1]^2$. This event can be written as the intersection of the events 
\begin{equation*}
B_n\coloneqq \{x\in F_n:\, x\text{ can be connected to } L \text{ and } R \text{ in } F_n\}.
\end{equation*}
Let $\theta(p)\coloneqq \mathbb{P}_p( F \text{ percolates})$. It is non-decreasing and upper
semi-continuous. We also say that arc percolation occurs in $F$ if there exists a continuous map
$\gamma: [0,1]\to[0,1]^2$, such that
\begin{enumerate}
\item $\gamma(0)\in L$ and $\gamma(1)\in R$,
\item $\gamma(t)\in F$ for all $t\in[0,1]$,
\item  $\gamma(t)\neq \gamma(s)$ for all $t\neq s$.
\end{enumerate}
Similarly, let $\theta^{\mathrm{arc}}(p)\coloneqq \mathbb{P}_p( \text{arc percolation occurs in } F
)$. Let us also introduce the critical probabilities
\begin{align*}
  p_d&\coloneqq \sup\{ p:\, \mathbb{P}_p( F \text{ is totally disconnected})=1\}; \\
  p_c&\coloneqq \inf\{ p:\, \theta(p)>0\}.
\end{align*}
Trivially $p_d\leq p_c$ and if $p_d=p_c$ then we write $p_0$ for the common value. The basic
connectivity properties of planar Mandelbrot percolation, due to Chayes \textit{et
al.}~\cite{Chayes2Durrett_FractalPercConnectivity_PTRF88},
Dekking--Meester~\cite{DekkingMeester_StructureFractalPerc} and
Meester~\cite{Meester_ConnectivityFractalPerc} can be summarised as follows.

\begin{theorem}[{\cite{Chayes2Durrett_FractalPercConnectivity_PTRF88, DekkingMeester_StructureFractalPerc,
  Meester_ConnectivityFractalPerc}}]\label{thm:basicConn}
For planar Mandelbrot percolation $p_d=p_c$ for all $N\geq 2$, and the common value $p_0\in(0,1)$.
Moreover, $\theta(p_0)>0$ and $\theta(p) = \theta^{\mathrm{arc}}(p)$ for all $p\in[0,1]$.

Conditioned on the event that $F\neq\varnothing$, almost surely the cardinality of $F^d$ is
uncountable regardless of $p$, furthermore, for $p\geq p_0$ the number of disjoint connected
components in $F^c$ is countable of which only finitely many intersect both $L$ and $R$.    
\end{theorem}
The regimes $p<p_0$ and $p>p_0$ are generally referred to as the sub- and super-critical phases. In
each phase different questions about the finer geometric properties of $F$ are more relevant.
Section~\ref{sec:ConnectedComp} concentrates on the dimension theory of $F^c$ and the regularity
properties of percolating paths in the super-critical regime, while~\cref{sec:Dust} expands more on
problems relevant for the dust phase, such as projections.

We now make some remarks on different aspects of~\cref{thm:basicConn}.

\begin{itemize}
\item 
The contents of~\cref{thm:basicConn} were recently generalised by Falconer and
Feng~\cite{FalconerFeng_FractalPercBMCarpet_ProcAMS25} from the classical $N\times N$ grid to the
statistically self-affine setting of the $N\times M$ subdivision of $[0,1]^2$. While some arguments
can be adapted quickly, one interesting question arises naturally. Due to symmetry, the critical
probability of left-right crossing is clearly the same as of top-bottom crossing in the $N\times N$
subdivision. It is not clear at all at first thought (at least not to the authors of the survey)
whether this should remain true for the $N\times M$ subdivision. Denoting
$\vartheta(p)\coloneqq\mathbb{P}_p(F \text{ percolates in top-bottom direction})$, the main result
of~\cite{FalconerFeng_FractalPercBMCarpet_ProcAMS25} is that $\theta(p)>0$ if and only if
$\vartheta(p)>0$, \ie the critical probabilities do coincide. However, the relationship between
$\theta(p)$ and $\vartheta(p)$ is unknown and posed as an open problem, see \cref{ques:falconerfeng}.
\item
  The question of determining $p_c$ is an unresolved difficult open problem, see \cref{ques:crit}.
  Except for special cases, there is no formal proof of the widely believed monotonicity
  property that $p_c(N')\leq p_c(N)$ for $N'\geq N$. On the plane for small values of $N=2,3,4$ the
  current best bounds for $p_c(N)$ are due to Don~\cite{Don_CriticalProbabilityFractalPerc}. A
  fairly simple branching process argument in~\cite{Chayes2Durrett_FractalPercConnectivity_PTRF88}
  shows that $p_c>N^{-1/2}$. A better lower bound for $N \geq 4$ follows
  from~\cite{Chayes2_FractalPercLargeN89} where it was shown that $p_c(N)$ tends to the percolation
  threshold of Bernoulli site percolation on the square lattice as $N\to\infty$ from above. Klatt
  and Winter~\cite{KlattWinter_GeometricFunctionalsFractalperc} investigated whether geometric
  functionals of Mandelbrot percolation could be used to estimate $p_c$.
\end{itemize}
The remaining remarks concern higher dimensional Mandelbrot percolation when $d\geq 3$.
\begin{itemize}
\item
In higher dimensions $\theta(p)$ can be defined the same way by modifying $L=\{0\}\times[0,1]^{d-1}$
and $R=\{1\}\times[0,1]^{d-1}$. The argument of~\cite{Chayes2Durrett_FractalPercConnectivity_PTRF88}
generalises to show that in arbitrary dimensions $p_c(N)\in(0,1)$ for all $N\geq 2$. The large $N$
limit result of~\cite{Chayes2_FractalPercLargeN89} was generalised by Falconer and
Grimmet~\cite{FalconerGrimmet_FractalPerc_92}. Instead of taking just the hypercubic lattice
$\mathbb{Z}^d$, additional edges are added, namely, two vertices of $\mathbb{Z}^d$ are connected if
and only if all their coordinates differ by at most one and are equal in at least one coordinate.
They show that $p_c(N)$ in dimension $d$ tends from above to the critical probability of site
percolation on this $d$-dimensional lattice.
\item
  It is conjectured that $p_d=p_c$ for all $N\geq 2$ in arbitrary dimensions, see \cref{conj:dustconnected}.
  This was shown by Broman and Camia~\cite[Theorem
4.3]{BromanCamia_UniversalConnectivityFractalPercModels} for all $d\geq 3$ and $N$ larger than some
$N_0(d)$.
\item
The discontinuity of $\theta(p)$ at $p_c$ is in contrast to independent percolation on
$\mathbb{Z}^2$ where the phase-transition is known to be continuous. However, this discontinuity is
a more general phenomenon. In higher dimensions Broman and Camia~\cite[Corollary
1.2]{BromanCamia_LargeNCrossingCriticalFractalPerc} proved that $\theta(p)$ is discontinuous at
$p_c(N)$ for $N$ larger than the same $N_0(d)$ as before.
In~\cite{BromanCamia_UniversalConnectivityFractalPercModels} it was also proved that
$\mathbb{P}_{p_d}(F \text{ is not totally disconnected})>0$  for all $d\geq 2$ and $N\geq 2$. 
\item
In dimensions $d\geq 3$ there are alternative options to define percolation. On the plane, if $F$
has a left-right crossing, then $[0,1]^2\setminus F$ does not have a top-bottom crossing. This is no
longer true in higher dimensions. Instead, one can consider $(d-1)$-dimensional ``sheets'' crossing
from $L$ to $R$ which block any top-bottom crossings in $[0,1]^d\setminus F$. Analogously,
$\theta_s(p)$ can be defined to be the probability of the event that sheet percolation occurs and
$p_s=p_s(N)=\inf\{p:\, \theta_s(p)>0\}$ to be the corresponding critical probability.
Orzechowski~\cite{Orzechowski_SheetPercInFractalPerc} proved a limit for $p_s(N)$ as $N\to\infty$
for any dimension $d$. This result implies that for $d\geq 3$ and $N$ large enough, $p_c(N)<p_s(N)$,
\ie there exists an interval of $p$ for which crossings are possible both in $F$ and its complement.
Similar to the discontinuity of $\theta(p)$ at $p_c$, Broman and Camia~\cite[Corollary
1.4]{BromanCamia_LargeNCrossingCriticalFractalPerc} showed that $\theta_s(p)$ is discontinuous at
$p_s$ for $N$ large enough in any dimension $d\geq 3$.
\end{itemize}

We end the section with some analogous results about fat fractal percolation
from~\cite{BromanEtal_FatFractalPerc}. We are in arbitrary dimensions and recall from
example~\cref{ex:fatperc} in~\cref{sec:generalConstr} that $p^{(n)}$ is a non-decreasing sequence
with $\prod_{n=1}^\infty p^{(n)}>0$ and $ p_{\bi}=p^{(n)}$ for all $\bi\in\Sigma_n$. It is immediate
that $\mathcal{L}(F)>0$ with positive probability, where $\mathcal{L}$ denotes the ($d$-dimensional)
Lebesgue measure. In fact more is true.

\begin{theorem}[{\cite[Theorem 1.8 \& 1.9]{BromanEtal_FatFractalPerc}}]\label{thm:FatPerc}
Conditioned on the event that $F\neq\varnothing$,
\begin{equation*}
\text{either }\;\; \mathcal{L}(F^d)=0 \text{ and } \mathcal{L}(F^c)>0 \text{ a.s.} \quad\text{ or
}\quad \mathcal{L}(F^d)>0 \text{ and } \mathcal{L}(F^c)=0 \text{ a.s.}
\end{equation*}
In particular, a sufficient condition for $\mathcal{L}(F^d)=0$ a.s. is if $\prod_{n=1}^\infty (p^{(n)})^{N^n}>0$.

Furthermore, $F$ has an empty interior a.s. if $\prod_{n=1}^\infty (p^{(n)})^{N^{dn}}=0$. Otherwise,
$F$ can be written as the union of finitely many cubes a.s.
\end{theorem} 
The authors of~\cite{BromanEtal_FatFractalPerc} point out that they do not know whether
$\mathcal{L}(F^d)>0$ and $\mathcal{L}(F^c)=0$ is possible at all and pose it as an open
problem, see \cref{ques:fatperc}. 

\clearpage
\section{Dimension theory of the limit set}\label{sec:DimLimitSet}
\subsection{Dimension theory}
Of particular interest in fractal geometry are the notions of dimensions, which quantify specific
scaling behaviour of sets and measures. 
The most well-known are the Hausdorff and box-counting dimensions, which are obtained by considering
optimal coverings with homogeneously sized covers (box-counting) or arbitrary countable covers
(Hausdorff). Briefly, they can be defined by
\[
  \dimh X = \inf\bigg\{\alpha > 0 : \text{$\forall\epsilon>0$ $\exists$ cover $U_i$ of $X$ with
  }\sum_{i\in\bbN} \diam(U_i)^\alpha < \epsilon\bigg\}
\]
and
\[
  \dimb X = \inf\left\{ \alpha>0 : \exists C>0 \forall r>0 \text{ such that }N_r(X) \leq
  C\left(\frac{1}{r}\right)^\alpha\right\},
\]
where $N_r(X)$ is the minimal cardinality of sets of diameter at most $r$ needed to cover $X$.
Note that technically we just defined the \emph{upper} box-counting dimension. It can also be
expressed as the upper limit $\limsup_{r\to 0} \log(N_r(X))/\log(1/r)$ and the \emph{lower}
box dimension can be defined analogously using the lower limit $\liminf$.
The usefulness of the Hausdorff dimension arises from the fact that it is associated with a measure,
the Hausdorff measure $\cH^s$, defined as
\[
  \cH^s(X) = \sup_{\delta>0}\inf\left\{\sum_{i\in\bbN} g(\diam(U_i)) : \text{$U_i$ is a cover of $X$ with
  }\diam(U_i)<\delta\right\},
\]
where $g(r) = r^s$.
In many cases the measure is positive and finite and if not, the gauge function $g$ can often be
chosen such that the measure becomes positive and finite.

A localised version of the upper box counting dimension also exists and is known as the Assouad
dimension. There is a wealth of literature on the subject and we refer the reader to
\cite{FraserBook,MackayTysonBook,RobinsonBook} for more
background and applications. 
Its importance comes from its link to embedding theory and the Assouad embedding theorem which
provides useful bounds on embeddability of metric spaces with finite Assouad dimension into
Euclidean space with minimal distortion.
The Assouad dimension is defined by
\[
  \dima X =\inf\left\{ \alpha>0 : \exists C>0 \;\forall 0<r<R\;\forall x\in X \text{ we have }N_r(B(x,R)\cap X) \leq
  C\left(\frac{R}{r}\right)^\alpha\right\} 
\]
and is linked to the metric property of \emph{doubling}: A metric space is doubling\footnote{A
  metric space is defined to be doubling if any ball $B(x,2r)$ can be covered by a uniformly bounded
number of balls $B(y_i,r)$.} if and only if it has finite Assouad dimension.  Over the last decade,
the study of \emph{interpolating dimensions} has gained a lot of traction, with the aim of providing
spectra that interpolate between various notions of dimension. Here we briefly define two important
such notions; the (generalised) Assouad spectrum and the intermediate dimension. We
refer the reader to \cite{FraserSurvey21} for more details. 
The basic idea is to interpolate between dimensions by restricting scale differences. The (upper) Assouad
spectrum is given by 
\[
  \dimaspec{\phi} X =\inf\left\{ \alpha>0 : \exists C>0 \;\forall 0<r<R^{1+\phi(R)}\;\forall x\in X
    \text{ we have }N_r(B(x,R)\cap X) \leq
  C\left(\frac{R}{r}\right)^\alpha\right\} 
\]
where the dimension function\footnote{A dimension function $\phi(R):(0,1)\to(0,\infty)$ is any
  function such that $\phi(R)\log(1/R)$ increases to infinity as
$R\to0$ and $\phi(R)$ decreases as $R\to0$.} $\phi(R)$ provides the separation between scales.
For the simple family of constant functions $\phi_\theta(R) = {1/\theta}-1$ with $\theta\in(0,1)$,
the resulting dimension is known
as the $\theta$-spectrum. The reader is referred to \cite{BanajiRutarTroscheit24} and references
therein for the most recent state-of-the-art for the generalised spectrum.  The Assouad spectra
interpolate between the upper box-counting and (quasi) Assouad dimensions and have been useful in
studying, amongst other, the spiral winding problem \cite{Fraser21_Spiral}, to give estimates on
quasi-conformal distortion \cite{ChrontsiosTyson23_distortion}, as well as allowing embedding
theorems using weaker assumptions than the Assouad embedding theorem
\cite{ChrontsiosTroscheit24_weak}.

\smallskip
A separate interpolation is that between the Hausdorff and the (upper) box-counting dimension, a
quantity known as the intermediate dimension.
It is defined similarly to the Hausdorff measure, but with the limitation that the scale difference
between covering sets cannot be too extreme. To see this, first note that the upper box-dimension is
defined through covers of equal sizes. The sum in the Hausdorff measure then becomes 
\[
  \sum_i \diam(U_i)^s = N_r(X)\cdot (2r^s).
\]
To define the upper $\Phi$-intermediate dimension~\cite{Banaji_GeneralisedIntermedDim} the maximal
and minimal covers are allowed to differ in size
$\min\diam(U_i) \geq \Phi(\max\diam(U_i))$, where $\Phi$ is a decreasing function. Then 
\begin{align*}
  \overline{\dim}^{\Phi} X = \inf\Big\{ \alpha>0 :
    \forall\eps>0,\,
    &\exists\delta_0\in(0,1],\,\forall\delta\in(0,\delta_0),\exists\text{cover
    }\{U_1,U_2,\dots\} \text{ of }X\\
    &\text{ such that } \Phi(\delta)\leq\diam(U_i)\leq\delta\text{ and }\sum_i
  \diam(U_i)^\alpha \leq \eps\Big\},
\end{align*}
where the special case $\Phi(\delta)=\delta^{1/\theta}$ for all $\delta\in[0,1]$ is the upper
$\theta$-intermediate dimension $\overline{\dim}_{\theta}X$ introduced
in~\cite{FalconerFraserKempton_IntermedDim}. The lower $\Phi$-intermediate dimension
$\underline{\dim}^{\Phi}$ is defined analogously by considering just a sequence
$\delta_{\varepsilon}\to 0$. If the two coincide, then the $\Phi$-intermediate dimension
$\dimPhi{\Phi}$ exists.  The study of the $\theta$-intermediate dimension, including its attainable
values \cite{BanajiRutar22_Attainable}, and their links to quasi-conformal distortion
\cite{FraserTysonSobolev} revealed its usefulness in metric geometry. For many sets its value has
been computed over the range $\theta\in[0,1]$ and highlights interesting geometric properties, see
for example~\cite{BanajiKoloss_IntermedDimBMCarpet,
BurrellFalconerFraser_EllipticalSpiral,FalconerFraserKempton_IntermedDim}.

\subsection{Mandelbrot percolation}
The Hausdorff dimension of Mandelbrot percolation was first computed by Hawkes
\cite{Hawkes_DimHFractalPerc81} and Kahane--Peyri\'ere \cite{KahanePeyriere_DimHFractalPerc76} to be
\[
  \dimh F = \frac{\log p N^d}{\log N} \qquad\text{almost surely, conditioned on }F\neq\varnothing,
\]
using martingale methods.
While not explicitly computed there, the corresponding result for the box-counting dimension, which
coincides with the Hausdorff dimension almost surely, can be derived from their covering arguments. 
In fact, given that the counting of covering sets is comparable to counting covers in $F_k$,
martingale convergence gives us the stronger result that the covering number $N_r$ is comparable to
$r^{-s}$, where $s = \dimb F$.

The Assouad dimension of fractal percolation, however, is more extremal. There are several ways of
seeing this, but all rely on the observation that the relationship between $r<R$ is arbitrary.
Letting $R\ll1$, there are many independent parts of the set in which the set can be ``thicker''
than average. The event of having all subcubes survive for $n$ generations is superexponentially
small in $n$, but still positive. Hence, a simple Borel-Canelli argument shows that there are
infinitely many local parts in Mandelbrot percolation where all subcubes are kept for any $n$, with
at least one point being kept in that subcube. This is enough to construct a sequence of blow-ups
that converge to the full unit cube and the Assouad dimension is the same as the ambient space
dimension, almost surely conditioned on non-extinction, that is,
\[
  \dima F = d \qquad\text{a.s.~ conditioned on non-extinction}.
\]
Using basic properties of the Assouad dimension this can be extended to
all orthogonal projections into $\bbR^k$ simultaneously, which are maximal. Almost surely,
conditioned on non-extinction,
\[
  \dima \Pi(F) = k, \quad\text{for all orthogonal projections }\Pi:\bbR^d\to\bbR^k\text{
  simultaneously.}
\]
This was first observed in Fraser--Miao--Troscheit
\cite{FraserMiaoTroscheit_DimARndFractals_ETDS18}, but the full Assouad dimension of Mandelbrot
percolation also follows from the earlier work of Berlinkov and J\"arvenp\"a\"a
\cite{BerlinkovJarvenpaa_FractalPercPorosity} using porosity arguments.
An interesting consequence of the full Assouad dimension result follows by observing that the
Assouad dimension is a bi-Lipschitz invariant. For arbitrary $d\in\bbN$ we may let $p$ be
arbitrarily close to $N^{-d}$ and thus the set has arbitrarily small Hausdorff and box-dimension
a.s.~conditioned on non-extinction. However, it can not be embedded into $\bbR^k$ for any $k<d$ with
bi-Lipschitz mappings as the Assouad dimension is full.

Other, more fine-grained questions have been answered. For instance, the Hausdorff and packing
dimensions come with natural measures: the Hausdorff and packing measure. Unlike the box-dimension
result, where the number of covers was comparable to $r^{-s}$, for arbitrary covers (and packings)
the measure at the right exponent is $0$ and $\infty$, respectively.
Modifying the measure slightly by allowing gauge functions, the correct gauge function can be
determined. This was done in full generality by Watanabe in \cite{Watanabe04_PackingMeasure} and
\cite{Watanabe07_ExactHausdorffMeasure}, where the gauge function $g(r) =
r^s \log|\log r|$ yields positive and finite Hausdorff measure. For the packing measure this is more
intricate and there is no gauge function for the percolation model. It is necessary that the number
of zero or one descendant does not occur almost surely.
See also Berlinkov and Mauldin \cite{BerlinkovMauldin02_PackingMeasure} for earlier work.

The observation that the Hausdorff and box-counting dimensions coincide, regardless of potential
overlaps has been long known for deterministic self-similar sets (and is false for self-affine), see
\cite{Falconer89_quasi}. In the random setting, these notions still tend to coincide, which was made formal in 
\cite{LiuWu02_dimensionalResult}, which was later extended to more subtle interactions involving
so-called graph-directed sets, that study subshifts of finite type, rather than the full coding
space \cite{Troscheit_DimHBRNDGraphDirectedSelfSim}. See also 
\cite{Olsen_Book, RoyUrbanski11_RandomGraphDirected} for related works on random graph-directed
sets.

Having established that the Hausdorff and box-counting dimensions coincide, this renders the
intermediate dimension moot: it will also just coincide. However, given the gap between box-counting
and Assouad dimension, we may investigate its spectrum. The spectrum turns out to be constant
\cite{Troscheit_QuasiADimStochSelfSim} and we need to use the generalised spectrum to determine the
point of transition. This was first attempted in \cite{Troscheit_AssouadSpectrumGWTree} and
completed in \cite{BanajiRutarTroscheit24}, where the exact transition function is determined and is
of the form
\(
  c\cdot{\log\log (1/R)}/{\log(1/ R)}.
\)
For this special dimension window we write $$\psi(R) = \log\log(1/R) / \log(1/R).$$

We summarise the dimension theory of Mandelbrot percolation in the following theorem.
\begin{theorem}\label{thm:DimMandelbrotPerc}
  Let $F$ be Mandelbrot percolation in $\mathbb{R}^d$ with parameters $N$ and $p>N^{-d}$.
  Let $\Pi:\mathbb{R}^d\to\mathbb{R}^k$ be an orthogonal projection for $1\leq k \leq d$. Then the
  following are simultaneously true for almost all realisations, conditioned on non-extinction:
  $$\dimh F = \dimb F = \frac{\log pN^d}{\log N}
  \quad\text{and}\quad
  \dima \Pi (F) = k,
  $$
  and in particular $\dima F = d$. Further,  
  \[
    \dimaspec{\phi} F = \alpha\frac{\log(1/p)}{d\log^2 N} + \frac{\log pN^d}{\log N}
  \]
  for all dimension functions $\phi$ such that $\lim_{R\to 0} \psi(R)/\phi(R) = \alpha$, where $\alpha\in[0,\log N^d]$.
\end{theorem}

We end our section by mentioning a recent result by Rossi and
Suomala~\cite{RossiSuomala_FractalPercQuasisymmMap_21IMRN}, who show that the dimension
of Mandelbrot percolation can be lowered by quasisymmetric mappings, establishing that the
quasi-conformal dimension is strictly smaller than the Hausdorff dimension.

\subsection{Stochastically self-affine sets}\label{sec:DimStochSelfAffine}

In this section we review some results on the dimension theory of fractal percolation coming from
diagonal self-affine IFSs. These planar carpet and higher
dimensional sponge models already have an extensive literature in the deterministic setting. For the
initial example, take the IFS consisting of those orientation preserving maps which take $[0,1]^2$
into the $M\cdot N$ many axis-parallel congruent rectangles of side-lengths $1/M$ and $1/N$ for some integers
$2\leq M<N$. We still assume that $p_{\bi}\equiv p$ for all $\bi\in\Sigma_*$ and call it fractal
percolation on the $M\times N$ grid. Deterministic IFSs on the $M\times N$ grid are usually referred
to as general Sierpi\'nski or Bedford--McMullen carpets. 

\begin{theorem}\label{thm:MNGridDimTheory}
Let $F$ be fractal percolation on the $M\times N$ grid with retention probability $p$. If $p\leq
1/(MN)$ then $F=\varnothing$ almost surely, otherwise, there is a positive probability that $F\neq
\varnothing$. Assuming that $1/(MN)<p\leq 1$, for almost all realisations, conditioned on non-extinction,
\begin{equation*}
\dim_{\mathrm{H}}F=\dim_{\mathrm{B}}F=
\begin{cases}
\log(pMN)/\log M, &\text{if } 1/(MN)<p\leq 1/N; \\
1+\log(pN)/\log N,  &\text{if } 1/N<p\leq 1,
\end{cases}
\end{equation*}
moreover, $\dim_{\mathrm{A}} F=2$.
\end{theorem}
These formulae follow from a collection of more general results
in~\cite{GatzourasLalley_DimHBStatSelfAffineBMCarpet, Troscheit_DimBRndBoxLikeSelfAffine,
BarralFeng_DimRndSelfAffineBMSponges, FraserMiaoTroscheit_DimARndFractals_ETDS18}. Some intuition
explaining these values. One can show that
$\dim_{\mathrm{H}}\proj_xF=\dim_{\mathrm{B}}\proj_xF=\min\{1,\log(pMN)/\log M\}$ almost surely,
conditioned on non-extinction, where $\proj_x$ denotes orthogonal projection to the $x$-axis. When
$p<1/N$, then on average less than 1 rectangle is chosen in each column, hence, it is reasonable to
expect that fibres above the $x$-axis do not ``carry'' any excess dimension. However, if $p>1/N$,
then a typical fibre ``behaves like'' a stochastically self-similar set with dimension
$\log(pN)/\log N$. As with many other random constructions, the Assouad dimension is full here as
well, see~\cite[Theorem 3.2]{FraserMiaoTroscheit_DimARndFractals_ETDS18}.

Gatzouras and Lalley considered a generalisation of fractal percolation
in~\cite{GatzourasLalley_DimHBStatSelfAffineBMCarpet}, commonly referred to as random substitutions,
see~\cref{sec:RndSubstitutions}, on the $M\times N$ grid. They determine the almost sure Hausdorff
and box dimension of these systems~\cite[Eq. (2.1) \&
(2.2)]{GatzourasLalley_DimHBStatSelfAffineBMCarpet} and provide an equivalent characterisation of
when the two dimensions
coincide~\cite[Proposition~6.1]{GatzourasLalley_DimHBStatSelfAffineBMCarpet}. Similarly to the
deterministic case, this can happen if the expected number of rectangles chosen in each column is
the same, resulting in uniform fibres ``on average''. In particular, fractal percolation fits this
description since all rectangles are retained with the same probability. Still on the plane,
Troscheit~\cite[Theorem 4.5]{Troscheit_DimBRndBoxLikeSelfAffine} computed the almost sure value of
the box dimension of fractal percolation on general box-like sets with rotations.

In higher dimensions, Barral and Feng~\cite{BarralFeng_DimRndSelfAffineBMSponges} consider fractal
percolation with equal retention probability on $N_1\times \ldots \times N_d$ grids with $2\leq
N_1<\ldots<N_d$. They prove that Mandelbrot measures, see \cref{sec:cascade}, are exact dimensional
and the dimension satisfies a Ledrappier--Young-type
formula~\cite[Theorem~2.2]{BarralFeng_DimRndSelfAffineBMSponges} which they then use to prove a
variational principle for the Hausdorff dimension of the limit set
$F$~\cite[Theorem~2.3]{BarralFeng_DimRndSelfAffineBMSponges}. They also give a formula for
$\dim_{\mathrm{B}} F$~\cite[Theorem~2.4]{BarralFeng_DimRndSelfAffineBMSponges} and give a
characterisation for when $\dim_{\mathrm{H}} F=\dim_{\mathrm{B}}
F$~\cite[Corollary~2.6]{BarralFeng_DimRndSelfAffineBMSponges}. In a very recent preprint, Barral and
Brunet~\cite{BarralBrunet2025} extend the variational principle for Hausdorff dimension using
inhomogeneous Mandelbrot measures to a significantly larger class of self-affine sponges. They also
show an analogous variational principle for the packing dimension.

\section{Geometry of the non-trivial connected components}\label{sec:ConnectedComp}

In this section we consider Mandelbrot percolation introduced in~\cref{sec:MandelbrotPerc} with
$p\geq p_d$. In this regime it follows from~\cref{thm:basicConn} that the union of the connected
components larger than a single point $F^c$ is non-empty almost surely conditioned on
non-extinction. Here we discuss results regarding the dimension theory of $F^c$ and the regularity
of the percolating paths in $F^c$. 

\subsection{Dimension theory}\label{subsec:DimConnectedComp}

Broman \textit{et al.}~\cite{BromanCamiaJoostenMeester_DimConnectedComponent} obtained the following
dimension results about $F^c$.

\begin{theorem}[{\cite[Theorem 1.1 $\&$ 1.2]{BromanCamiaJoostenMeester_DimConnectedComponent} }]\label{thm:DimConnComp}
For any $d\geq 2$, $N\geq 2$ and $p\geq p_d(N)$, almost surely
\begin{equation*}
	\dim_{\mathrm{B}}  F^c = \dim_{\mathrm{B}}  F =\dim_{\mathrm{H}}  F.
\end{equation*}
If $ F \neq \varnothing$ then almost surely
\begin{equation*}
	\dim_{\mathrm{H}}  F^c<\dim_{\mathrm{H}} F = \dim_{\mathrm{H}}  F^d.
\end{equation*}
Moreover, for every $p$ there exists $1 \leq \beta=\beta(p) \leq d$ such that
\begin{equation*}
	\mathbb{P}_p\left( F^c=\varnothing \text { or } \dim_{\mathrm{H}}  F^c=\beta\right)=1.
\end{equation*}
\end{theorem}

Recall from~\cref{thm:basicConn} and the remarks following that $p_d=p_c$ for $d=2$ and also for $N$
large enough in higher dimensions, furthermore,  $\mathbb{P}_{p_d}(F \text{ is not totally
disconnected})>0$  for all $d\geq 2$ and $N\geq 2$. The interesting phenomena here is the strict
dimension gap between the Hausdorff and box dimension of $F^c$ or another way of putting it is that
the Hausdorff dimension of the `dust' component $F^d$ is strictly larger than the Hausdorff
dimension of $F^c$. We note that the dimension gap also follows
from~\cite[Remark~11]{ChenOjalaRossiSuomala_FractalPercPorosity} using porosity arguments. Here we
sketch the proof of the argument from~\cite{BromanCamiaJoostenMeester_DimConnectedComponent}.

\begin{proof}[Sketch of proof]
We begin with the proof of $\dim_{\mathrm{B}} F^c = \dim_{\mathrm{B}} F$. The condition $p\geq
p_d(N)$ implies that almost surely either $F=\varnothing$ (in which case the claim is trivial) or
$F^c\neq \varnothing$. Distance between subsets $A,B\subset \mathbb{R}^d$ is defined the usual way
$d(A,B)\coloneqq \inf \{|x-y|:\, x\in A, y\in B\}$, where $|\cdot|$ denotes the Euclidean distance.
Start from a $\delta$-cover $\{B_i\}_{i=1}^{M_{\delta}}$ of $F^c$ using the minimal number
$M_{\delta}$ of closed cubes of side length $\delta$. For any point $x\in F$, if
$d(x,\bigcup_{i=1}^{M_{\delta}} B_i)=0$, then $x\in\bigcup_{i=1}^{M_{\delta}} B_i$ since all the
$B_i$ are closed. Assuming that $\{B_i\}_{i=1}^{M_{\delta}}$ is not a cover of $F$, there exists an
$x\in F$ such that $d(x,\bigcup_{i=1}^{M_{\delta}} B_i)>0$. Then there also exists a retained cube
$\pi(\bi)= f_{\bi}([0,1]^d)$ with $\bi\in Q_n$ for some $n$ large enough such that
$d(\pi(\bi),\bigcup_{i=1}^{M_{\delta}} B_i)>0$. Using stochastic self-similarity and the fact that
$F^c\neq \varnothing$ imply that $F^c\cap \pi(\bi)\neq \varnothing$. This contradicts that
$\{B_i\}_{i=1}^{M_{\delta}}$ is a $\delta$-cover of $F^c$, hence, $\{B_i\}_{i=1}^{M_{\delta}}$ is
actually also an optimal $\delta$-cover of $F$ as well which shows that $\dim_{\mathrm{B}} F^c =
\dim_{\mathrm{B}} F$.

Now let us turn to the proof of $\dim_{\mathrm{H}}  F^c<\dim_{\mathrm{H}} F$. Assuming $N\geq 5$ is
odd (the proof adapts to other values of $N$ as well), for any $\bi\in\Sigma_n$ and odd $k\leq N$,
let $B(\pi(\bi),kN^{-n})$ denote the cube concentric to $\pi(\bi)$ with side lengths $kN^{-n}$. The
argument relies on a result~\cite[Corollary~2.6]{BromanCamia_UniversalConnectivityFractalPercModels}
about the crossing probability of ``shells''. In our context, let $i\in\Sigma_1$ be such that
$(1/2,\ldots,1/2)\in \pi(i)$ and let us introduce $\varphi_{N,d}(p)$ to be the probability that
there exists a component in $F^c$ which crosses the shell $[0,1]^d\setminus B(\pi(i),3N^{-1})$. The
crucial result from~\cite{BromanCamia_UniversalConnectivityFractalPercModels} is that
$\varphi_{N,d}(p)>0$ whenever $p\geq p_d(N)$. Using statistical self-similarity this bound can be
``transferred'' to the crossing of a shell $B(\pi(\bi),N^{-n+1})\setminus B(\pi(\bi),3N^{-n})$ for
any $\bi\in\Sigma_n$ and $n\geq 1$. Note that $B(\pi(\bi),3N^{-n})$ and/or $B(\pi(\bi),N^{-n+1})$
need not be completely contained in $[0,1]^d$, but this does not alter the argument. 

For $\varepsilon>0$, let $F^{c,\varepsilon}$ denote the connected components in $F$ whose diameter
is at least $\varepsilon$. A simple consequence of the definition of Hausdorff dimension is that
$\dim_{\mathrm{H}}F^c=\sup _{\varepsilon} \dim_{\mathrm{H}}F^{c, \varepsilon}$. The proof is then
complete by the claim that 
\begin{equation}\label{eq:40}
	\dim_{\mathrm{H}}F^{c, \varepsilon} \leq d+\frac{\log \left(p \varphi_{N, d}(p)\right)}{\log
	N}<\dim_{\mathrm{H}} F \quad \text { a.s. }
\end{equation}
independent of $\varepsilon$. Hence, it remains to show~\cref{eq:40}. 

To motivate the specific cover of $F^{c,\varepsilon}$, let us make a few observations. For any
$\bi\in\Sigma_n$, the intersection $\pi(\bi)\cap (B(\pi(\bi),N^{-n+1})\setminus
B(\pi(\bi),3N^{-n}))=\varnothing$, therefore, if we define
\begin{equation*}
\mathcal{W}_n\coloneqq \{\pi(\bi):\, \bi\in Q_n \text{ and there is a component in } F^c \text{
crossing }  B(\pi(\bi),N^{-n+1})\setminus B(\pi(\bi),3N^{-n})\},
\end{equation*}
then by statistical self-similarity and independence between the two conditions determining whether
$\pi(\bi)\in\mathcal{W}_n$, wee see that
\begin{equation*}
\mathbb{P}_p\left( \pi(\bi)\in\mathcal{W}_n \right) \leq p \varphi_{N, d}(p)
\end{equation*} 
independent of $n$. The inequality is due to the fact that $B(\pi(\bi),3N^{-n})$ and/or
$B(\pi(\bi),N^{-n+1})$ need not be completely contained in $[0,1]^d$. Also note by construction that
two shells
\begin{equation*}
\big(B(\pi(\bi),N^{-n+1})\setminus B(\pi(\bi),3N^{-n})\big) \cap
\big(B(\pi(\bi|_{n-1}),N^{-n+2})\setminus B(\pi(\bi|_{n-1}),3N^{-n+1})\big) =\varnothing
\end{equation*}
for any $n\geq 2$, $\bi\in\Sigma_n$ where $\bi|_{k}=i_1,\ldots,i_k$ (for $k\leq n$). This implies
that the events $\pi(\bi)\in\mathcal{W}_n$ and $\pi(\bi|_{n-1})\in\mathcal{W}_{n-1}$ are
independent. 

Assuming that $F^c\neq \varnothing$, we fix $\varepsilon>0$ small enough such that
$F^{c,\varepsilon}\neq\varnothing$ and let $\ell=\ell(\varepsilon)$ be such that $N^{-\ell+1}\leq
\varepsilon/d$. For $n\geq \ell$ and $\bi\in\Sigma_n$, the key observation is that
\begin{equation*}
	\text{if }\; \pi(\bi) \cap F^{c,\varepsilon}\neq\varnothing \;\text{ then }
	\;\pi(\bi|_m)\in\mathcal{W}_m \text{ for all } m=\ell,\ell+1,\ldots,n.
\end{equation*}
Hence, we obtain a cover $\mathcal{V}_n$ of $F^{c,\varepsilon}$ for $n\geq \ell(\varepsilon)$ by setting
\begin{equation}\label{eq:41}
\mathcal{V}_n\coloneqq \{\pi(\bi):\, \bi\in\Sigma_n \text{ and } \pi(\bi|_m)\in\mathcal{W}_m \text{
for all } m=\ell,\ell+1,\ldots,n\}.
\end{equation}
Using our observations
\begin{equation*}
\mathbb{P}_p\left(\pi(\bi) \in \mathcal{V}_n\right) = \mathbb{P}_p\left(\bigcap_{m=l}^n \pi(\bi|_m)
\in \mathcal{W}_m\right) 
=\prod_{m=l}^n \mathbb{P}_p\left(\pi(\bi|_m) \in \mathcal{W}_m\right) 
\leq\left(p \varphi_{N, d}(p)\right)^{n-l+1}.
\end{equation*}
Let $\#\mathcal{V}_n$ denote the cardinality of $\mathcal{V}_n$. Using the cover $\mathcal{V}_n$ of
$F^{c,\varepsilon}$ we bound
\begin{equation*}
	\begin{aligned}
		\mathbb{E}_p\left(\mathcal{H}^s\left(F^{c, \varepsilon}\right)\right) & \leq \liminf
		_{n \rightarrow \infty} \mathbb{E}_p\bigg(\sum_{\pi(\bi) \in \mathcal{V}_n}
		|\mathrm{diam}\, \pi(\bi)|^s\bigg) \\
		& =\liminf _{n \rightarrow \infty}\big(\sqrt{d} N^{-n}\big)^s\,
		\mathbb{E}_p\left(\#\mathcal{V}_n\right) \\
		& \leq \liminf _{n \rightarrow \infty} d^{s / 2} N^{-s n}N^{dn}\big(p \varphi_{N,
		d}(p)\big)^{n-l+1} \\
		& =d^{s / 2} \big( p \varphi_{N, d}(p) \big)^{1-\ell} \lim _{n \rightarrow \infty}
		N^{n\big(d+\frac{\log (p \varphi_{N, d}(p))}{\log N}-s\big)}.
	\end{aligned}
\end{equation*}
Note that $\ell$ is fixed, so the proof of~\cref{eq:40} is complete since the limit is finite if and only if
\begin{equation*}
	s \geq d+\frac{\log \left(p \varphi_{N, d}(p)\right)}{\log N}.
\end{equation*}
For a proof of the final claim, we refer to~\cite{BromanCamiaJoostenMeester_DimConnectedComponent}.
\end{proof}

\begin{remark}
Since the cover $\mathcal{V}_n$ consists of cubes of the same diameter, one can easily argue that
the bound obtained in~\cref{eq:40} is also an upper bound for $\mathbb{E}_p\big(
\underline{\dim}_{\mathrm{B}} F^{c,\varepsilon} \big)$. So it is the ``small components'' of $F^c$
that actually determine its box dimension. However, $\dim_{\mathrm{B}}F^c\neq\sup _{\varepsilon}
\dim_{\mathrm{B}}F^{c, \varepsilon}$, hence, $\dim_{\mathrm{B}}F^c$ does not drop below
$\dim_{\mathrm{B}}F$. 
\end{remark}

\subsubsection{Assouad spectrum and intermediate dimensions of \texorpdfstring{$F^c$}{F^c}}

In light of~\cref{thm:DimConnComp} one could hope to gain more nuanced geometric insight about $F^c$
by considering its Assouad spectrum and intermediate dimensions. Here we prove that the Assouad
dimension of the connected component is maximal (\ie equal to the ambient space dimension $d$) while
the Assouad spectrum and $\theta$-intermediate dimension are trivial and equal to the box dimension of the dust.
We are not aware of this having been attempted elsewhere.
\begin{theorem}\label{thm:DimAssouadSpectrumConnComp}
Almost surely and conditioned on non-extinction, for any dimension function $\phi$, 
\begin{equation*}
  \dim_{\mathrm{A}}^{\phi} F^c = \dim_{\mathrm{A}}^{\phi} F.
\end{equation*}
\end{theorem}
Recall from~\cref{thm:DimMandelbrotPerc} the value of $\dim_{\mathrm{A}}^{\phi} F$. In particular,
for $\phi(R) = 1/\theta-1$ with $\theta\in (0,1]$, the $\phi$-Assouad spectrum is the
$\theta$-Assouad spectrum and it coincides with $\dim_{\mathrm{H}} F=\dim_{\mathrm{B}} F$.
\begin{proof}
  Given a dimension function $\psi$, we always have the inequality $\dim_{\mathrm{A}}^{\psi} F^c
  \leq
  \dim_{\mathrm{A}}^{\psi} F$ by monotonicity of the generalised Assouad dimension.
  Further, by \cite[Corollary 2.2]{BanajiRutarTroscheit24} and \cite[Corollary
  H]{BanajiRutarTroscheit24}, it remains to prove that
  \[
    \dim_{\mathrm{A}}^\phi F^c \geq \alpha \frac{\log(1/p)}{d\log^2 N}+\frac{\log p N^d}{\log N}
  \]
  for $\phi(R) = (1/\alpha)\log\log(1/R)/\log(1/R)$, where $\alpha\in(0,\log N^d]$. This will also
  follow from \cite[Theorems F \& G]{BanajiRutarTroscheit24} with a minor modification. We will
  sketch only where the proof differs from that of Theorem G and the reader is referred to
  \cite{BanajiRutarTroscheit24} for the full derivation of the dimension of $F$.

  Let $\theta(p)>0$ be
  the probability that the limit set contains a connected component. 
  Let $m=p N^d$ be the expected number of descendants and let $\gamma>1$ be the unique exponent such
  that $m^\gamma=N^{d}$.
  Let $1<t<\gamma$ and let $E_k$ be the event that more than $m^{t n(k)}$ subcubes $C_i$ are retained at
  level $n(k)$, of which at least $\theta(p)/2\cdot m^{t n(k)}$ many contain a connected component. 
  By \cite[Theorem F]{BanajiRutarTroscheit24} and Hoeffding's inequality, for arbitrarily small $\eps>0$,
  \begin{align*}
    P_k &:=\bbP(E_k) 
    \gtrsim
    \exp\left(-m^{(t-1+\eps)\tfrac{\gamma}{\gamma-1}n(k)}\right)
    \cdot \bbP\left(\#\{C_i\} \geq \theta(p)/2\cdot m^{t n(k)}\right)
    \\
	&\geq\exp\left(-m^{(t-1+\eps)\tfrac{\gamma}{\gamma-1}n(k)}\right)
	\cdot \left(1-\exp\left(\tfrac12 \theta(p)^2 m^{t n(k)}\right)\right)
	\\
	&\gtrsim\exp\left(-m^{(t-1+\eps)\tfrac{\gamma}{\gamma-1}n(k)}\right),
  \end{align*}
  where $\gtrsim$ refers to the lower bound holding up to a uniform multiplicative constant.
  Now let $n(k) = (1/\alpha) \log k$ and $t$ be such that $(t-1+\eps)/\alpha\tfrac{\gamma}{\gamma-1} \log m <1$.
  We get
  \begin{align*}
    m^k P_k &\gtrsim
    m^k \exp \left(-m^{(t-1+\eps)/\alpha\tfrac{\gamma}{\gamma-1}\log k}\right)
    \\
	&=
    \exp\left(k \log m -k^{(t-1+\eps)/\alpha\tfrac{\gamma}{\gamma-1}\log m}\right)
    \\
	&\gtrsim \exp\left(k \log m\right) = m^k.
  \end{align*}
  Since $m^k P_k$ increases exponentially in $k$, we can use a Borel-Cantelli lemma for trees, see
  \cite[Lemma 3.6 (ii)]{BanajiRutarTroscheit24} and obtain that $E_k$ occurs infinitely often almost
  surely, conditioned on non-extinction, as a subtree in $F$. By construction, we further have these
  subtrees are also in $F^c$.
  But then there exist $m^{tn(k)}$ subcubes in $F^c$ at level $k+n(k)$ with side-length $N^{-(k+n(k))}$.
  Since further $n(k) = \log k$, we get that the $\phi$-Assouad dimension of $F^c$ is at least $\log
  (m^t)/\log N$.
  Now our condition on $t$ gives
  \[
    1>\frac{t-1+\eps}{\alpha}\frac{\gamma}{\gamma-1} \log m
  \]
  and so we can choose
  \[
    t=\alpha \frac{\gamma -1}{\gamma}\frac{1}{\log m}+1-2\eps 
  \]
  and
  \begin{align*}
    \dim_{\mathrm{A}}^\phi F^c 
    &\geq \frac{\log m}{\log N}\left(\alpha \frac{\gamma
    -1}{\gamma}\frac{1}{\log m}+1-2\eps \right)
    \\
    & = \alpha \frac{\log(1/p)}{d\log^2 N}+\frac{\log p N^d}{\log N}-2\eps\frac{\log pN^d}{\log N}.
  \end{align*}
  But as $\eps$ was arbitrary, we get the required conclusion.
\end{proof}
We now show that $\dim_{\theta} F^c$ is also
trivial and coincides with the dimension of the
dust for all $\theta\in(0,1]$. It is possible to construct sets with constant $\theta$-intermediate
dimension if $\dim_{\theta}X\in\{0,\dim_{\mathrm{L}}X,\dim_{\mathrm{A}}X\}$, where
$\dim_{\mathrm{L}}X$ denotes the lower dimension of $X$. To the best of our knowledge this is the
first example which does not fall in this mentioned category, yet the $\theta$-intermediate
dimension is constant. 
\begin{theorem}\label{thm:DimIntermediateSpectrumConnComp}
For any $d\geq 2$, $N\geq 2$ and $p\geq p_d(N)$, almost surely and conditioned on non-extinction, 
\begin{equation*}
\dim_{\mathrm{H}}F^c=\dim_{\theta=0} F^c  < \dim_{\theta} F^c = \dim_{\mathrm{B}} F^c \;\text{ for every } \theta\in(0,1].
\end{equation*}
\end{theorem}
\begin{proof}	
Define $\delta\coloneqq \dim_{\mathrm{B}} F^c=d+\frac{\log p}{\log N}$, fix $\theta\in(0,1)$ and let
$K\in\mathbb{N}$
which indexes the level of the construction. Since $\overline{\dim}_{\theta} F^c\leq \delta$, it is
enough to show that $\underline{\dim}_{\theta}F^c\geq \delta$. Recall the notation of $F_K$
from~\cref{eq:11}. The
objective is to define a measure $\mu_K$ as the sum of point masses supported on a carefully chosen
subset of $F_K$, to which we can apply the mass distribution
principle~\cite[Proposition~2.2]{FalconerFraserKempton_IntermedDim} for intermediate dimensions.
Choose $0<\varepsilon_0\ll N^{-K}$. Recall $F^{c,\varepsilon_0}\subset F^c$ is the subset with
connected components of diameter at least $\varepsilon_0$. To cover $F^{c,\varepsilon_0}$ with level
$\theta K$ squares (when talking about levels, always interpret a real number as its lower integer
part), we use the cover $\mathcal{V}_{\theta K}$ defined in~\cref{eq:41}. In expectation,
$\#\mathcal{V}_{\theta K}$ is bounded from above by $C\cdot N^{\theta K (\delta+\log
\varphi_{N,d}(p) / \log N)}$. Therefore, the number of level $\theta K$ squares needed to cover
$F^c\setminus\mathcal{V}_{\theta K}$ is still bounded from below by $C\cdot N^{\theta K
(\delta-\alpha)}$ for some $\alpha>0$ which tends to 0 as $K\to\infty$ since
$\dim_{\mathrm{B}}F^c=\delta$ almost surely. Note that throughout the proof all uniform constants
are simply denoted by $C$ even though they may be different.

Define
\begin{equation*}
\mathcal{F}_{\theta K}\coloneqq \{ \pi(\bi):\, \bi\in\Sigma_{\theta K} \;\;\&\;\; \pi(\bi)\cap
F^c\setminus\mathcal{V}_{\theta K} \neq \varnothing \;\;\&\;\; \pi(\bi)\cap \mathcal{V}_{\theta K} =
\varnothing\},
\end{equation*}
\ie those level $\theta K$ squares which are needed to cover $F^c\setminus\mathcal{V}_{\theta K}$
but do not share a boundary with any square in $\mathcal{V}_{\theta K}$. Squares in
$\mathcal{F}_{\theta K}$ may still intersect each other on their boundary. We further discard
squares from $\mathcal{F}_{\theta K}$ so that in $\mathcal{F}_{\theta K}$ we are left with disjoint
squares with maximal cardinality. We still have $\#\mathcal{F}_{\theta K}\geq C\cdot N^{\theta K
(\delta-\alpha)}$. The choice of $\varepsilon_0$ and statistical self-similarity imply that in
distribution $F^c\cap \pi(\bi)$ are all independent rescaled Mandelbrot percolations conditioned to
have non-trivial connected components of (strictly positive) diameter at most $\varepsilon_0\cdot
N^{\theta K}$. Independence comes from the disjointness of the squares and that $\varepsilon_0\ll
N^{-K}$.   

We claim that for every $p\geq p_d$ and $\eta>0$,
\begin{equation}\label{eq:42}
\mathbb{P}_p(\text{largest component of } F^c \text{ has diameter} \in(0,\eta]) >0.
\end{equation}
Indeed, define $L_0$ to be the smallest integer so that $N^{-L_0}\leq\eta/\sqrt{d}$. Consider the
configuration where $F_{L_0}$ consists of a single level $L_0$ square. This configuration has
positive probability. Recall that $\mathbb{P}_{p_d}(F \text{ is not totally disconnected})>0$ for
all $d\geq 2$ and $N\geq 2$. Therefore, by statistical self-similarity, within this one square there
is strictly positive probability of having a non-trivial connected component, which shows that the
probability in~\cref{eq:42} is indeed positive.

Hence, claim~\cref{eq:42} and~\cref{thm:DimConnComp} imply that almost surely
$\dim_{\mathrm{B}}(\pi(\bi)\cap F^c)=\delta$ for all $\pi(\bi)\in\mathcal{F}_{\theta K}$. In
particular, at least $C\cdot N^{(1-\theta) K(\delta-\alpha)}$ level $K$ squares are needed to cover
$F^c$ within each square $\pi(\bi)\in\mathcal{F}_{\theta K}$. As a result, the total number of level
$K$ squares needed to cover $F^c\cap \mathcal{F}_{\theta K}$ is bounded from below by $C\cdot
N^{K(\delta-\alpha)}$. Denote these level $K$ squares by $\widetilde{\mathcal{F}}_K$.

The goal is to put point masses on a positive proportion of the squares in
$\widetilde{\mathcal{F}}_K$ which have the following `uniformly expected branching' property. Let $\gamma>0$ and $\psi\in(\theta,1)$. We say that
$\pi(\bi)\in \widetilde{\mathcal{F}}_K$ is $(\gamma,\psi)$-good if 
\begin{equation*}
\#\big\{\pi(\bj)\in\widetilde{\mathcal{F}}_K:\, \pi(\bj)\subset \pi(\bi|_{L})\big\} \leq c\cdot
N^{(K-L)(\delta+\gamma)} \text{ for every level } L\in[\theta K, \psi K]
\end{equation*}
for some fixed constant $c\geq 1$. Denoting $A_K\coloneqq\big\{ \pi(\bi)\in
\widetilde{\mathcal{F}}_K:\, \pi(\bi) \text{ is } (\gamma,\psi)\text{-good} \big\}$, we claim that
\begin{equation}\label{eq:43}
\#A_K \geq C\cdot N^{K(\delta-\alpha)}
\end{equation}
for some random $C$ which, conditioned on non-extinction, is bounded away from $0$ almost surely.

Indeed, let $Z_k$ be a finitely supported Galton--Watson process with mean offspring $\bbE(Z_1)=N^\delta$.
For all $\gamma>0$ and $c_1>0$, a Chernoff bound gives
\[
  \bbP\big(Z_k \geq N^{(\delta+\gamma)k}\big) =
  \bbP\big(\exp(c_1 Z_k) \geq \exp(c_1 N^{(\delta+\gamma)k})\big) \leq \exp(-c_1 N^{(\delta+\gamma)k})\cdot
  \bbE\big(\exp(c_1 Z_k)\big).
\]
By \cite[Theorem 4]{Athreya94}, see also the discussion in \cite[Section
3]{Troscheit_QuasiADimStochSelfSim}, there exists a particular value of $c_1$ for which the
expectation $\sup_k \bbE\big(\exp(c_1 Z_k)\big)\leq c_2<\infty$. Since the retained squares
(starting at any level) form a Galton--Watson process with mean offspring $N^{\delta}$, we obtain
the bound 
\[
  \bbP_p\big(\#\big\{\pi(\bj)\in\widetilde{\mathcal{F}}_K:\, \pi(\bj)\subset \pi(\bi|_{L})\big\} >
  N^{(\delta+\gamma)(K-L)}\big)
  \leq c_2\cdot \exp(-c_1 N^{(\delta+\gamma)(K-L)})
\]
for every $\pi(\bi)\in\widetilde{\mathcal{F}}_K$ and $L<K$.
Using a union bound, the existence of a not $(\gamma,\psi)$-good
$\pi(\bi)\in\widetilde{\mathcal{F}}_K$ can be bounded from above by 
\begin{align*}
  &\bbP_p\big(\exists \pi(\bi)\in\widetilde{\mathcal{F}}_K,\;\exists L\in [\theta K,\psi K] \text{ such that
  }\#\big\{\pi(\bj)\in\widetilde{\mathcal{F}}_K:\, \pi(\bj)\subset \pi(\bi|_{L})\big\} > N^{(\delta+\gamma)(K-L)} \big)
  \\
  &\phantom{x}\leq \sum_{\bi\in\Sigma_K}\sum_{L=\theta K}^{\psi K} c_2\cdot \exp(-c_1 N^{(\delta+\gamma)(K-L)})
  \leq N^{dK} \sum_{L=\theta K}^{\psi K} c_2\cdot \exp(-c_1 N^{(\delta+\gamma)(K-L)})
  \\
  &\phantom{x}\leq C N^{dK} \exp(-c_1 N^{(\delta+\gamma)(1-\psi)K})
\end{align*}
which is summable in $K$. Hence, almost surely, the bound \cref{eq:43} is satisfied for large enough
$K$ or small enough  positive constant $C$, almost surely.

Having established~\cref{eq:43}, we now define the measure
\begin{equation*}
	\mu_K(\cdot) \coloneqq \sum_{\pi(\bi)\in A_K} N^{-K(\delta-\alpha)} \chi_{x_{\pi(\bi)}}(\cdot),
\end{equation*}
where $\chi_{x_{\pi(\bi)}}$ is the indicator function at the point $x_{\pi(\bi)}\in\pi(\bi)\cap
F^c$. We want to apply the mass distribution
principle~\cite[Proposition~2.2]{FalconerFraserKempton_IntermedDim} with the measures $\mu_K$.
Using~\cref{eq:43} we have $\mu_K(F^c)\geq C>0$. We need to bound $\mu_K(U)$ from above for any set
$U$ with diameter $|U|\in[N^{-K},N^{-\theta K}]$.

First assume that $|U|=N^{-\eta K}\in[N^{-K},N^{-\psi K}]$. Since $U$ can intersect at most $C\cdot
N^{d\cdot (1-\eta)K}$ many points in the support of $\mu_K$, we can bound
\begin{equation*}
\mu_K(U) \leq C\cdot N^{d\cdot (1-\eta)K}\cdot N^{-K(\delta-\alpha)} =
C\cdot|U|^{(\delta-\alpha)/\eta-(1/\eta-1)d} \leq C\cdot|U|^{\delta-\alpha-(1/\psi-1)d}.
\end{equation*}
On the other hand, if $|U|=N^{-\eta K}\in(N^{-\psi K},N^{-\theta K}]$, then using the fact that all
$\pi(\bi)\in A_K$ are $(\gamma,\psi)$-good, we can bound
\begin{equation*}
\mu_K(U) \leq C\cdot N^{(1-\eta)K(\delta+\gamma)} \cdot N^{-K(\delta-\alpha)} = C\cdot
|U|^{\delta-\gamma(1/\eta-1)-\alpha/\eta} \leq C\cdot |U|^{\delta-\gamma(1/\theta-1)-\alpha/\theta}.
\end{equation*}
Now applying~\cite[Proposition~2.2]{FalconerFraserKempton_IntermedDim}, we have shown that
\begin{equation*}
\underline{\dim}_{\theta}F^c\geq \min\big\{ \delta-\alpha-(1/\psi-1)d, \delta-\gamma(1/\theta-1)-\alpha/\theta \big\}.
\end{equation*}
The parameters $\alpha$ and $\gamma$ can be made arbitrarily close to $0$ while $\psi$ arbitrarily
close to $1$ by choosing $K$ larger to begin with. Hence, we conclude that
$\underline{\dim}_{\theta}F^c\geq \dim_{\mathrm{B}} F^c$.
\end{proof}

When the $\theta$-intermediate dimension is discontinuous at $\theta=0$, Banaji showed
in~\cite[Theorem 6.1]{Banaji_GeneralisedIntermedDim} that for any compact set $X$ that there is a
family of functions $\{\Phi_s:\, s\in[\dim_{\mathrm{H}}X,\overline{\dim}_{\mathrm{B}}X]\}$ for which
the $\Phi$-intermediate dimensions interpolate all the way between $\dim_{\mathrm{H}}X$ and
$\overline{\dim}_{\mathrm{B}}X$. One can ask for which threshold function this occurs, see
\cref{ques:theta} and \cref{conj:dimfunct}.
Covering the really small non-trivial connected components efficiently is the challenging part. Even
partial answers could provide explicit bounds on the unknown value of $\beta(p)$
in~\cref{thm:DimConnComp}.

\subsection{Regularity of percolating paths}

Meester~\cite[Theorem~3.1]{Meester_ConnectivityFractalPerc} showed that arc percolation and
percolation are probabilistically equivalent to each other for planar Mandelbrot percolation. We now
elaborate on what is known about the regularity properties of these percolating paths.

A natural initial question is whether a percolating path can be the graph of a (monotone) function.
Let $\Gamma:[0,1]\to[0,1]^2$ be a continuous path denoted $\Gamma(t)=(\gamma_1(t),\gamma_2(t))$. We
say that the image $\Gamma[0,1]$ is a directed crossing of $[0,1]^2$ if $\gamma_1(0)=0$, the
function $\gamma_1$ is non-decreasing and $\gamma_1(1)=1$. If in addition $\gamma_2$ is also
non-decreasing, then we call $\Gamma[0,1]$ a north-east oriented crossing. The term `stiff' is also
used for directed percolation. It turns out that directed percolation is not possible.

\begin{theorem}[{\cite{Chayes_NoDirectedFractalPerc95,
  ChayesPemantlePeres_DirectedFractalPerc97}}]\label{thm:OrientedPerc}
Consider planar Mandelbrot percolation with level dependent retention probabilities, \ie there is a
sequence $p^{(n)}$ such that $ p_{\bi}=p^{(n)}$ for all $\bi\in\Sigma_n$. If $\prod_{n=1}^\infty
p^{(n)}=0$, then almost surely $F$ does not contain a directed crossing of $[0,1]^2$.
\end{theorem}  
Recall that if $\prod_{n=1}^\infty p^{(n)}>0$, then the Lebesgue measure of $F$ is strictly positive
almost surely. The theorem was initially proved by Chayes~\cite{Chayes_NoDirectedFractalPerc95} for
$ p^{(n)}\equiv p$ and subsequently strengthened to the stated form by Chayes \textit{et al.}
in~\cite{ChayesPemantlePeres_DirectedFractalPerc97}. The main idea of the proof is to extend the
percolation process to the half-strip $[0,1]\times[0,\infty)$ by ``stacking'' independent copies of
Mandelbrot percolation on top of each other and constructing a `contour' using discarded squares
which blocks any north-east oriented path starting from the bottom $[0,1]\times \{0\}$ that
``escapes'' to the top. More formally, a sequence of discarded squares $R_1,\ldots, R_M$ form a
contour over $[0,1]$ if 
\begin{enumerate}
\item the right edge of $R_i$ lies on the same (vertical) line as the left edge of $R_{i+1}$ for
  each $i=1,\ldots, M-1$; 
\item the left edge of $R_1$ is on the line $x=0$ while the right edge of $R_M$ is on the line
  $x=1$;
\item the top of $R_{i+1}$ is not strictly below the (horizontal) line determined by the bottom of
  $R_i$ for each $i=1,\ldots, M-1$.
\end{enumerate} 
The first two items ensure that every vertical line $\{x\}\times[0,\infty)$ is blocked by some $R_i$
while the last one prevents a north-east oriented path escaping. The height of the contour is simply
the maximum height of the bottoms of the squares in the contour. The key technical
observation~\cite[Proposition~2.1]{ChayesPemantlePeres_DirectedFractalPerc97} is that the expected
value of the height of a contour gets arbitrarily small (\ie close to a horizontal line) as a
contour uses smaller and smaller squares from higher levels. Using this, a counting argument
precludes nearly horizontal directed crossings. 

The absence of directed crossings and statistical self-similarity strongly suggest that a
percolating path has dimension strictly larger than one. Indeed, one formulation of the result of
Chayes~\cite{Chayes_BoxDimCrossingLowerBound96} is that the minimum number of level $k$ retained
squares required to connect $L$ with $R$ scales as $N^{k(1+\zeta)}$ for some $\zeta>0$ if and only
if directed percolation is not possible. Hence, it follows from~\cref{thm:OrientedPerc} that the
lower box dimension of any percolating path must be strictly bigger than one, although no dependence
of $\zeta$ on $p$ was exhibited. Later Orzechowski~\cite{Orzechowski_LowerBoundDimBPercolatingPath}
with a much more direct approach gave an explicit lower bound. Namely, there exists a constant
$v(N)>0$ such that almost surely every percolating path $\gamma$ in $F$ satisfies
\begin{equation*}
\underline{\dim}_{\mathrm{B}} \gamma \geq 1+ v(N)\cdot \frac{(1-p)^4}{|\log (1-p) |^3}.
\end{equation*}
This bound tends to $1$ as $p\to1$. An immediate corollary of this is that almost surely, no
percolating path can be $\alpha$-H\"older continuous for any $\alpha> \big( 1+ v(N)\cdot
(1-p)^4|\log (1-p) |^{-3} \big)^{-1}$. In particular, $F$ does not contain any rectifiable curves.
Orzechowski~\cite{Orzechowski_PhDThesis} also showed that for $N\geq 3$ and $p\geq 1-M^{-5}/15$
almost surely there exists a percolating path in $F$ with upper box dimension bounded from above by
an explicit function which tends to 1 as $p\to 1$. 

The limit set $F$ does not contain $\alpha$-H\"older percolating paths for $\alpha$ close to one,
but Broman \textit{et al.} \cite[Theorem 1.4 \&
1.5]{BromanCamiaJoostenMeester_DimConnectedComponent} showed using the sophisticated machinery of
Aizenman and Burchard~\cite{AizenmanBurchard_RndCurves99} that in the plane $F^c$ is almost surely
the union of non-trivial H\"older continuous curves, all with the same exponent. Another consequence
of their approach is that all these continuous curves have Hausdorff dimension strictly larger than
1. 

These result were recently supplemented by Buczolich \textit{et al.}
\cite{BuczoEtal_FractalPercUnrectifiable} who considered a notion of fractional rectifiability.
Given $0<\alpha\leq 1$, a set $E\subset \mathbb{R}^d$ is purely $\alpha$-unrectifiable if
$\mathcal{H}^{\frac{1}{\alpha}}(E\cap \gamma([0,1]))=0$ for all $\alpha$-H\"older curves
$\gamma:[0,1]\to\mathbb{R}^d$. Their main result~\cite[Theorem
6.13]{BuczoEtal_FractalPercUnrectifiable} is that in arbitrary dimension $d\geq 2$, for every $0\leq
p<1$ and $N\geq 2$, there exists $\alpha_0<1$ such that almost surely the limit set of Mandelbrot
percolation is purely $\alpha$-unrectifiable for all $\alpha_0<\alpha\leq 1$. In particular, this
also implies that $F$ is almost surely purely $k$-unrectifiable for $k\in\mathbb{N}\setminus\{0\}$,
\ie $\mathcal{H}^k(M\cap F) =0$ for all $k$-dimensional $C^1$-submanifolds $M\subset \mathbb{R}^d$.

\section{Metric geometry of fractal percolation}
\label{sec:Dust}

\subsection{Projections, Slices}\label{sec:projections}
Of particular interest in geometric measure theory are the sizes of projections and slices of sets and
measures.
A fundamental classical result is the Marstrand projection theorem \cite{Marstrand54} in $\mathbb{R}^2$,
which states that almost every orthogonal projection of a set with Hausdorff dimension $s>1$ has positive
Lebesgue measure. In general, almost every projection of a set with Hausdorff dimension $s>0$ has
Hausdorff dimension $\min\{1,s\}$, with similar results holding in higher dimensions. 
Finding conditions on which these results can be improved is a very active research area and beyond
the scope of this survey.
We will just briefly mention that Mandelbrot and fractal percolation are models under which
Marstrand's projection theorem can be significantly improved. We refer the reader to the recent
survey by  Orgov\'anyi and Simon~\cite{OrgovanyiSimon_Survey25} that covers questions on the
projections of stochastically self-similar processes in detail as well as the surveys
\cite{RamsSimon_FractalPercSurvey14,SimonVago_FractalPercSurvey18,FalconerJinSurvey}.

A sampling of results on projections of percolation is listed below.
\begin{itemize}
  \item Sharp Marstrand projection: For Mandelbrot percolation in the plane, Rams and Simon
    \cite{RamsSimon_ProjFractalPerc_JStatPhys14,RamsSimon_ProjFractalPerc_ETDS15} show that if
    $\dimh F <1$, then almost surely the orthogonal projections onto \emph{all} lines have the same
    dimension as $F$, and if $\dimh F>1$, then for \emph{any} smooth real valued function $f$ that
    is strictly increasing in both coordinates, the image $f(E)$ contains an interval. This is, in
    particular, true for any orthogonal projections that are not the principal directions.

\item Fractal percolation: Dekking and Meester \cite{DekkingMeester_StructureFractalPerc} considered
  random fractal percolation of the Sierpi\'nski carpet and showed that there are six possible
    phases of connectedness. These phases of connectedness do not appear in this form for 
    Mandelbrot percolation. See also \cite[Section 3a \& Theorem 10]{Chayes_FGSSurvey95}.

\item Maximal slices: Recently, Shmerkin and Suomala \cite{ShmerkinSuomala_LargestSliceFractalPerc} studied the
  question of when Mandelbrot percolation of the square contains $k\in\bbN$ colinear points. In particular, they
  find that for $p\leq 2^{(-k-2)/k}$, the percolation set $F$ does not contain $k$ colinear points
  almost surely, but if $p> 2^{(-k-2)/k}$ then, almost surely and conditioned on non-extinction,
  there exists a line $\ell$ such that $\#(\ell\cap F) \geq k$. They answer a question posed in
  their earlier article \cite{ShmerkinSuomala_Patterns20} which investigates patterns in random
  constructions, including percolation. 

\end{itemize}

\subsection{Visible parts of percolation}
Related to projections are the \emph{visible parts} of a set $X$. Let $X\subset \bbR^d$ be compact
and let $\ell$ be a line that does not intersect $X$. 
The visible part of $X$ from $\ell$ are all the points in $X$ that can be ``seen'' from $\ell$, 
\ie the orthogonal projections onto the line from $x\in X$ do not hit any other point in $X$.
Formally, let $\Pi_\ell(x)$ be the orthogonal projection of $x$ onto $\ell$ and let $[x,y]$ be the
line segment between $x,y\in \bbR^d$. The visible set of $X$ from $\ell$ is
\[
  V_{\ell}(X) = \{x\in X : [x,\Pi_{\ell}(x)]\cap X = \{x\}\}.
\]
Similar to the Marstrand projection theorem, we expect a stronger statement for Mandelbrot
percolation to hold. In fact, Arhosalo \textit{et al.}~\cite{ArhosaloJ2RamsShmerkin_VisiblePartsFractalPerc}
proved in $\bbR^2$, for $N=2$ and $p$ sufficiently large that $\dimh F > 1$, \ie $p>\tfrac12$,
that
\[
  \dimh V_\ell(F) = \dimb V_{\ell}(F) = 1
\]
almost surely, conditioned on non-extinction.
Moreover, $\cH^1(F) \in (0,\infty)$ almost surely, conditioned on non-extinction.

\subsection{Porosity}
A separate view of studying the geometry of sets is through its porosity.
We consider a set $X$ as a subset of an ambient space, say $\bbR^d$ and denote the complement of a
set $X$ in this space by $\overline{X}$. The porosity of $X$ at $x$ of
scale $r>0$, written $\por(X,x,r)$ is the size of the maximal `hole':
\[
  \por(X,x,r) = \sup\left\{\rho>0 : \exists y\in \overline{X}\text{ such that }B(y,\rho)\subseteq
  B(x,r)\cap \overline{X}\right\}.
\]
This quantity can take any value in $[0,\tfrac12]$ and we are interested in the `overall porosity'
of points in $X$. To achieve this one can consider the upper and lower porosity of $X$ at $x$,
defined by
\[
  \underline{\por}(X,x) = \inf_{r>0}\por(X,x,r)
  \qquad\text{and}\qquad
  \overline{\por}(X,x) = \sup_{r>0}\por(X,x,r).
\]
The study of porosity is a very active field in metric geometry and we refer the reader to
\cite{ZajicekSurveyOld, ZajicekSurveyNew,ShmerkinPorosity} for 
surveys on porosity in general. 
Unsurprisingly, there are many connections between the dimension theory of sets and its porosity. A
particular example is that the existence of a point $x\in X\subseteq\bbR^d$ with lower porosity $0$
implies that the Assouad dimension of $X$ is maximal, \ie is $d$.
The porosity of Mandelbrot percolation has been studied, amongst others in 
Berlinkov--J\"arvenp\"a\"a \cite{BerlinkovJarvenpaa_FractalPercPorosity}
and
Chen et al. \cite{ChenOjalaRossiSuomala_FractalPercPorosity}.
In particular, Mandelbrot percolation $F$ has upper and lower porosity $\tfrac12$ and $0$, almost
surely conditioned on non-extinction, at almost every point in $F$.
In particular, the points for which this is not the case have strictly smaller Hausdorff dimension
$s$ of the full percolated set $F$:
\[
  \dimh\{x\in F : \underline{\por}(F,x) > \eps\} < s-\delta(\eps)
  \quad\text{and}\quad
  \dimh\{x\in F : \overline{\por}(F,x) < \tfrac12 - \eps\} < s-\delta(\eps).
\]
It would be interesting to see whether there are notions that work in tandem with interpolating
dimensions and whether the gauge functions necessary there agree with those in the study of the
intermediate and Assouad spectrum.

\subsection{Minkowski content and other curvatures}
\label{sec:Minkowski}
Another approach in studying the geometry of sets in an ambient space is by investigating the size
of its $\eps$ neighbourhood as $\eps\to 0$.
For compact subsets of $\bbR^d$ this is usually achieved by measuring the decay of the Lebesgue
measure. We write
\[
  \cM^s(X) = \limsup_{\eps\to 0} \eps^{s-d} \cL^d([X]_{\eps})
\]
for the (upper) $s$-Minkowski content, where $[X]_\eps=\{y\in \bbR^d : \inf_{x\in X}|x-y|\le\eps\}$
is the closed $\eps$-neighbourhood of $X$.
Similarly, one can define the lower Minkowski content and \textit{the}
Minkoswki content if both agree. 
It is worth pointing out here that there exists a critical threshold $s_0$ from which the content
jumps from $0$ to $\infty$. This threshold coincides with the upper and lower box-counting dimension
and is responsible for their alternative name, the Minkowski dimension.

Analysing Mandelbrot percolation, one may assume that the Minkowski content at the right exponent is
zero, as is the case for the Hausdorff measure. However, that is not the case. Let $s = \dimh F$. A martingale argument
shows that the number of covering sets of size $r$, suitably resized, converges to a random number
$r^{s}N_r(F) \to C$ that is positive and finite, conditioned on non-extinction.
This already shows that the upper and lower Minkowski contents are positive and finite when relating
the cover to $r$-neighbourhoods. A more sophisticated renewal argument shows that the Minkowski
content almost surely exists. This is also true for self-similar fractal percolation under some
additional mild conditions \cite{Gatzouras99_Lacunarity}.
Introducing some averaging with the Minkowski content by setting
\[
  \overline{\cM}^s(X) = \lim_{\delta\to 0} \frac{1}{|\log\delta|}\int_{\delta}^1
  \eps^{s-d}\cL^d([X]_\eps)\frac{1}{\eps}d\eps
\]
these issues are resolved and the average Minkowski content exists and is positive and finite for
all stochastically self-similar sets. Surprisingly, this does not hold when the model is changed to
a skew system, or $1$-variable system, see \cref{sec:vvar}.

Extensions to the Minkowski content to more general curvature measures have also been considered,
and we refer the reader to \cite{Winter08_CurvatureMeasures} and \cite{Zahle11_Curvature} and
references therein.

\subsection{Intersection with Brownian motion}
Recall that Brownian motion $W_t$ $(t\geq 0)$ in $\bbR^d$ has Hausdorff dimension $2$ for all $d\geq
2$. See \cite{Taylor1953} for the original proof, or \cite[Theorem 16.5]{FalconerFractalGeometry}
for a modern exposition.
Now consider Mandelbrot percolation in $\bbR^d$ for $d\geq3$ and let $p$ be such that $p= N^{2-d}$. In
particular, the dimension of $F$ is $2$ almost surely, conditioned on non-extinctions and $F$ is
totally disconnected for large $d$ and as such dust-like. 
However, Peres~\cite{Peres_BrownianPathIntersectFractalPerc} showed that the percolation set and Brownian motion
$W_t$
are intersection equivalent, even if the dust and Brownian motion have entirely different
connectivity (and topological) properties. That is, he showed that
for all open sets $O\subseteq[0,1]^d$, 
the probability $\bbP(F\cap O=\varnothing)$ of dust intersecting $O$ is positive if and only if 
$\bbP(\{W_t : t>0\}\cap O \neq \varnothing)>0$.

Other such fascinating intersection results exist and from a number theoretical perspective,
Mandelbrot percolation, or fractal percolation in general, is
large and intersects badly approximable points. In fact, the set intersects any hyperplane winning
set, see \cite{Dayan_GWFractalWinningSets22}.

\section{Generalisations}\label{sec:Generalisations}

The fractal percolation process can be generalised in numerous directions. We give just a small
sample here that is far from being complete.

\subsection{Weakening spatial independence}
A strong assumption that is key to the study of Mandelbrot and fractal percolation is the spatial
independence of different subsets. There are various ways of weakening that assumption, one of them
being the introduction of a ``finite variety'' of arrangements that can be observed at a specific
construction level. Depending on the question these models can often be studied using (multitype)
branching processes in random environments. 

\subsubsection{Random homogeneous model or the $1$-variable model}\label{sec:vvar}
Consider, as in the Mandelbrot percolation case the subdivision of a cube into $N^d$ congruent
subcubes. Instead of independently deciding at construction level $n$ how the surviving subcubes
percolate, we decide the subcube structure once and replicate these subsets in all subcubes of that
level. In the language introduced earlier, we have $X_{\bi} = X_{\bj}$ for all $\bi,\bj\in\Sigma_k$
whenever $i_k=j_k$. This model is sometimes called the \emph{random homogeneous model} or the $1$-\emph{variable
model}, since there is only $1$ possible arrangement of subcubes at every construction level.
This specific model has attracted a lot of attention with the dimension theory of general
self-similar fractal percolation being studied in \cite{Troscheit_DimHBRNDGraphDirectedSelfSim}.
The results further extends to graph directed constructions where $\Sigma_k$ is replaced by a
subshift of finite type, \ie there exist forbidden transitions.
The self-affine dimension theory of these constructions was also studied in
\cite{GuiLi08_randomMcMullen} for random Bedford-McMullen carpets and
\cite{Troscheit_DimBRndBoxLikeSelfAffine} for more general carpets. Fraser and Olsen
\cite{FraserOlsen11_MultifractalSpectra} further studied the multifractal spectra of random
Sierpi\'nski sponges. We also refer the reader to \cite{Olsen_Book} for more general random
self-similar graph directed sets.

\subsubsection{$V$-variable model or random code trees with necks}
The terminology of $1$-variable random sets arises as it is a special case of the $V$-\emph{variable model}
introduced by Barnsley, Hutchinson, and Stenflo~\cite{BarnsleyHutchinsonStenflo2012}. Here, the
construction mechanisms allows for up to $V\in\bbN$ distinct subtree constructions, with the
stochastically-self similar case being $V=\infty$, where there are no restrictions on the number of
sub-arrangements.
The $V$-variable construction is in and of itself a special case of \emph{random code trees with necks},
see~\cite{Jarvenpaaetal14_codetrees}, where the self-affine dimension theory was investigated. The model imposes no
restrictions on the number of sub-arrangements at every construction level, but requires there to be
a distribution such that there are infinitely many construction levels at which every subtree is
identical and such that this distribution is invariant under a shift from these ``neck'' levels.
Interestingly, this defining mechanism is enough for there to be no gauge function such that the
random sets have positive but finite Hausdorff measure \cite{Troscheit2021_CodeTrees}. Further, the
Minkowski content $\cM$ and the averaged Minkowski content $\overline{\cM}$ (see
\cref{sec:Minkowski}) also do not exist
\cite{Troscheit23-Minkowski} almost surely.

\subsubsection{Random substitutions}\label{sec:RndSubstitutions}

Let $\mathbf{q}$ be a probability distribution on the subsets of the index set $\Sigma_1$. If
$\pi(\bi)\subset F_n$ for some $\bi\in\Sigma_n$, then instead of choosing independently for each
$\pi(\bi j)$ whether to keep it or not, choose the whole configuration $\{\pi(\bi j):\,
j\in\Sigma_1\}$ according to $\mathbf{q}$ and repeat independently for each $\pi(\bi)\subset F_n$.
This is a \emph{random substitution} system which clearly generalises fractal percolation and was
considered in for
example~\cite{DekkingMeester_StructureFractalPerc,Falconer1986_RandomFractals,
FalconerGrimmet_FractalPerc_92,GatzourasLalley_DimHBStatSelfAffineBMCarpet}.

\subsubsection{Overlapping IFSs}
In all the previous examples it was assumed that $f_{i}\big((0,1)^d\big)\cap
f_{j}\big((0,1)^d\big)=\varnothing$ for all $i\neq j\in\Sigma_1$, \ie the IFS satisfies the
so-called \emph{open set condition}. There are several ways to weaken this condition and introduce
overlaps into the system. This is an extensively studied area in the deterministic setting, however,
apart from a preprint of Orgov\'anyi and Simon~\cite{OrgovanyiSimon_InteriorPointsLebesgueMeasure},
we are unaware of any further works in this direction in the fractal percolation setting. They show
the existence of attractors which have positive Lebesgue measure, but empty interior almost surely
conditioned on the attractor not being empty.  
\vspace{2mm}

\noindent Finally, we mention that Dekking and Wal~\cite{DekkingWal_FractalPercCellularAutomata} 
introduced branching cellular automata to allow for neighbour interaction in Mandelbrot percolation.

\subsection{Cascade measures and random geometry}\label{sec:cascade}
In this survey we mostly focussed on the dimension theory of sets arising from a fractal percolation
process. Similar questions can be asked for measures as well. Consider the basic construction of
Mandelbrot percolation, but instead of removing with probability $1-p$, we associate to each coding an
independent identically distributed random weight $X_v\geq 0$ that satisfies $\mathbb{E}(X_v)=1$ and
construct a measure on cubes by setting 
\[
  \mu_n (Q) =
  X_{\emptyset}\cdot X_{i_1}\cdot X_{i_2} \dots X_{i_k} \cdot \mathcal{L}^d(Q)
\]
for cubes $Q$ of level $k\le n$ with coding $i_1\dots i_k$, where $\mathcal{L}^d$ is the
$d$-dimensional Lebesgue measure.
Under mild assumptions on the random variable $X$, the weak limit $\mu = \lim_{n\to\infty} \mu_n$
exists as the measure has a martingale structure. The limiting measure $\mu$ is then a positive and
finite measure, called the multiplicative cascade measure. 

This measure is the natural measure analogue to Mandelbrot percolation and much research has gone
into studying its measure-theoretic and dimension theoretic properties. Further, they have been used
as simple models that generate random geometry in one dimension.
We refer the reader to
\cite{BarralFeng_ProjPlanarMandelbrotMeasure,BarralJin_MultiplicativeCascades_IMRN22,BarralSeuret07_Renewal,
FalconerJin14,HamblyJones03_ThickThinPoints,PatschkeZahle90_RandMeasure,PeresRams_ProjNatMeasureFractalPerc16} for
more background on the dimension theory of the measures and its projections.
For their use in one-dimensional random
geometry and links of the KPZ equation and dimension of images under such random geometries, we
refer the reader to~\cite{BarralEtAl14,BenjaminiSchramm09,FalconerTroscheit23}.

\subsection{Spatially independent martingales}\label{sec:generalProcesses}
A detailed
study of ``fractal'' and stochastically self-similar measures was studied in
\cite{ShmerkinSuomala_SIMartingales} by Shmerkin and Suomala, who coined the term
\emph{spatially independent martingales} for such constructions. They investigated many structures
found in these random measures and their survey \cite{ShmerkinSuomala_FARFSurvey17} summarises their
main ideas. 
In summary they
\begin{itemize}
  \item Establish strong projection results for random sets and measures, \cf
    \cref{sec:projections}.,
  \item prove strong dimension preservation laws on intersections and slices,
  \item give smoothness results on convolutions of such random measures,
  \item deduce results on sumsets,
  \item and obtain Fourier decay results.
\end{itemize}
Most of these results are beyond the scope of this survey, but Shmerkin and Suomala show that many
results on fractal percolation can be translated into much more general frameworks of stochastically
self-similar results.


\section{Open Problems}\label{sec:OPenProblems}

Despite the simplicity of the model, Mandelbrot percolation has proven to possess a myriad of
interesting properties. Plenty of open problems remain to guarantee that it remains an active area
of research for the foreseeable future. 

\subsection{Mandelbrot percolation}
Many open questions remain, even for the original Mandelbrot percolation setting and have been open
for some time. 
The exact critical threshold is unknown:
\begin{question}\label{ques:crit}
What is the critical percolation threshold $p_c$ for Mandelbrot percolation with parameters $N,d$?
\end{question}

What is more, it still remains unknown in general whether the critical percolation threshold is the
same as the threshold for the dust phase:
\begin{conjecture}\label{conj:dustconnected}
  For Mandelbrot percolation with parameters $N\geq 2$ and $d\geq 1$, we have $p_d=p_c$.
\end{conjecture}

Varying the probabilities in the construction of fat percolation, the statement of
\cref{thm:FatPerc} gives two possibilities. It is unknown whether both cases can occur and there is
a dichotomy, or whether just one of them holds:
\begin{question}\label{ques:fatperc}
Is it possible that $\mathcal{L}(F^d)>0$ and $\mathcal{L}(F^c)=0$ for fat Mandelbrot percolation?
\end{question}

The finer dimension theoretic analysis of the connected components
through intermediate dimensions and spectra, further opens up
the question at which exact threshold function one sees the interpolation. We ask:
\begin{question}\label{ques:theta}
What family of functions $\{\Phi_s:\, s\in[\dim_{\mathrm{H}}F^c,\overline{\dim}_{\mathrm{B}}F^c]\}$
recover the full interpolation between $\dim_{\mathrm{H}}F^c$ and $\overline{\dim}_{\mathrm{B}}F^c$?
\end{question}
In particular, studying the intermediate dimension with coverings bounded between size $\delta$ and
$\Phi(\delta) = \delta^{1+\phi(\delta)}$, we ask which function $\phi$ observes the interpolation and whether it coincides
with the family of functions interpolating the Assouad spectrum 
\(
  \psi(R) = \frac{\log |\log R|}{|\log R|}.
\)
In fact, given that the estimates in \cref{thm:DimIntermediateSpectrumConnComp} use similar methods as used for the Assouad spectrum
interpolation in \cite{BanajiRutarTroscheit24}, we conjecture that 
\begin{conjecture}\label{conj:dimfunct}
  There exists a monotone function $\alpha(s)\colon[0,1]\to[0,\infty)$ such that 
  \[
    \{\dim^{\Phi} F^c \colon \Phi(\delta) = \delta^{1+\alpha(s)\psi(\delta)}\} =
    [\dim_{\mathrm{H}}F^c,\overline{\dim}_{\mathrm{B}}F^c].
  \]
\end{conjecture}

\subsection{Planar fractal percolation}
Finally, we wish to promote one possible general direction. Roughly
speaking, how universal are facts about Mandelbrot percolation in the context fractal
percolation (as defined in~\cref{sec:generalConstr})? For the most part, papers going beyond the
self-similar IFS satisfying the open set condition have concentrated
on the dimension of the limiting set, but many further questions could be asked. In the
deterministic setting, such general IFSs (like self-affine) have exhibited numerous phenomena that the
self-similar setting does not.  

A few concrete questions could be the following:

\begin{question}\label{ques:falconerfeng}
  For $M\times N$ percolation, what is the relationship between the probabilities of a left-right
  crossings $\theta(p)$ and a top-bottom crossing $\vartheta(p)$?
\end{question}
\begin{question}
Does directed percolation exist for fractal percolation on the $M\times N$ grid?
\end{question} 
Moving away from the $M\times N$ grid, one
may consider general grid-line subdivisions. These correspond to stochastically self-affine sets.
While their dimension theory has, to a limited extend, been analysed, nothing is known about their
connectivity properties.
We ask:
\begin{question}
  Under what conditions are the critical thresholds for left-right and top-bottom percolation the
  same?
\end{question}
\begin{question}
  What is the minimal level of complexity necessary for the left-right and top-bottom percolation
  thresholds to be distinct?
\end{question}
In particular, is it necessary to construct fractal percolation on a proper
self-affine subset? 

\smallskip
For the dimension gap between the connected component and the dust, we ask whether there is the same
gap:
\begin{question} Let $F$ be the limit set of general self-affine percolation. Assume that $p$ is
  large enough for there to exist connected components. Is it true that 
  \[\hdim F^c < \bdim F^c =\bdim F^d?\]
  Is it also possible that
  \[
    \hdim F^c < \hdim F^d < \bdim F^d?
  \]
\end{question}

\bigskip
\subsection*{Acknowledgement}
IK's research is supported by the European Research Council Marie
Curie--Sk\l{}odowska Personal Fellowship \#101109013, the Hungarian NRDI Office grant K142169 and
the J\'anos Bolyai Research Scholarship of the Hungarian Academy of Sciences. 

ST's research was initially supported by the
European Research Council Marie Curie--Sk\l{}odowska Personal Fellowship \#101064701. ST acknowledges
financial support from the Lisa \& Carl--Gustav Esseen's fund for mathematics (2025).

\printbibliography

\end{document}